\documentclass[10pt]{article}

\usepackage[utf8]{inputenc}
\usepackage[T1]{fontenc} 
\usepackage[english]{babel}

\usepackage{amsmath,amsthm,amsfonts,amssymb,amscd}
\usepackage{enumerate}
\usepackage{latexsym}
\usepackage{mathrsfs}
\usepackage{bbm}
\usepackage{xcolor}
\usepackage{tcolorbox}
\usepackage{arydshln} 
\usepackage{stmaryrd}
\usepackage[makeroom]{cancel}
\usepackage{lineno} 
\usepackage{tcolorbox}
\usepackage{hyperref}
\usepackage{authblk}
\newcommand{\email}[1]{\href{mailto:#1}{#1}}

\usepackage{geometry}
\usepackage{hyperref}
\hypersetup{
    colorlinks = true,
    linkcolor={blue},
    citecolor={blue},
    urlcolor={blue},
}

\usepackage[capitalize,nameinlink]{cleveref}

\makeatletter
\def\thm@space@setup{%
  \thm@preskip=1ex plus .25ex minus .1ex
  \thm@postskip=1ex plus .25ex minus .1ex
}
\makeatother

\theoremstyle{plain}
\newtheorem{theorem}{Theorem}[section]
\newtheorem{lemma}[theorem]{Lemma}
\newtheorem{corollary}[theorem]{Corollary}
\newtheorem{proposition}[theorem]{Proposition}
\newtheorem{definition}[theorem]{Definition}

\theoremstyle{remark}
\newtheorem{remark}[theorem]{Remark}

\newtheorem{assu}{Assumption}

\usepackage{titlesec}

\titleformat{\section}
  {\normalfont\Large\bfseries}
  {\thesection}
  {0.5em}
  {}

\newenvironment{keywords}
  {\par\smallskip\noindent\textbf{Keywords: }}
  {\par}
\newenvironment{MSCcodes}
  {\par\smallskip\noindent\textbf{MSC codes: }}
  {\par}

\newcommand{\norm}[1]{\left\| #1 \right\|}
\newcommand{\R}{\mathbb{R}}
\newcommand{\N}{\mathbb{N}}

\renewcommand{\P}{\mathscr{P}}

\newcommand{\supp}{\operatorname{supp}}
\newcommand{\AC}{\operatorname{AC}}
\newcommand{\id}{\operatorname{id}}

\newcommand{\Lip}{\operatorname{Lip}}
\newcommand{\PT}{\operatorname{PT}}
\newcommand{\PP}{\mathsf{P}}
\renewcommand{\Bar}{\operatorname{Bar}}

\newcommand{\Ptwo}{\mathscr P_2(\mathbb R^d)}

\newenvironment{blue}{\begingroup\color{blue}}{\endgroup}

\newcommand{\mcJ}{\mathcal{J}}
\newcommand{\mcG}{\mathcal{G}}

\newcommand{\mcF}{\mathcal{F}}
\newcommand{\mcM}{\mathcal{M}}
\newcommand{\mcB}{\mathcal{B}}

\newcommand{\mcA}{\mathcal{A}}

\newcommand{\mcO}{\mathcal{O}}

\newcommand{\mcW}{\mathcal{W}}
\newcommand{\mcU}{\mathcal{U}}

\newcommand{\sP}{\operatorname{\mathscr{P}}}
\newcommand{\sL}{\mathscr{L}}

\newcommand{\vphi}{\varphi}
\renewcommand{\epsilon}{\varepsilon}

\newcommand{\mbP}{\mathbb{P}}

\providecommand{\norm}[1]{\lVert#1 \rVert} 
\providecommand{\inner}[1]{\langle#1 \rangle} 
\newcommand{\Tanreg}{\operatorname{Tan}^\text{reg}}
\newcommand{\bfTan}{\mathbf{Tan}}

\definecolor{color0}{rgb}{0,0,0}
\definecolor{color1}{rgb}{0.22,0.45,0.70}
\definecolor{color2}{rgb}{0.45,0.45,0.45}

\title{Pontryagin maximum principle for non-smooth state-constrained control problems over Wasserstein spaces}
\author{
  Fernanda Urrea\thanks{Normandie Univ, INSA Rouen Normandie, Laboratoire LMI, France. Corresponding author: \email{fernanda.urrea@insa-rouen.fr}.} \quad
  Ernesto Treumún\thanks{Unité de Mathématiques Appliquées, ENSTA,  Institut Polytechnique de Paris, Palaiseau, France. \email{ernesto.treumun@ensta.fr}.}
}

\date{}

\begin{document}


\maketitle

\bigskip

\begin{abstract}
    We study a state-constrained Bolza optimal control problem governed by a non-local continuity equation in the Wasserstein space $\sP_2(\R^d)$. We develop a Pontryagin maximum principle for these problems in which both the cost functionals and the constraints are only locally Lipschitz, and the initial measure does not necessarily have compact support. To derive the main result, we introduce a notion of a Clarke-type deterministic subdifferential over $\sP_2(\R^d)$. The proof of the main result combines a careful adaptation of the Lasry-Lions regularization to the Wasserstein space coupled with stability results of the Clarke subdifferential that allow us to recover the first-order information. We also show that, for functionals convex along plans, this subdifferential admits a global supporting characterization and provides a first-order characterization of global minimizers.
\end{abstract}



\section{Introduction}

We study a state-constrained optimal control problem in Bolza form over the Wasserstein space $\sP_2(\R^d)$. More specifically, we study

\begin{equation*}
    \begin{aligned}
        \min_{u \in\mathcal{U}_{ad}} \quad  &\mcJ(\mu^u(T)) + \int_0^T L(t, u(t), \mu^u(t))\,\mathrm{d}t \\
        \text{s.t}\quad&\left\{
          \begin{array}{l}
          \partial_t\mu^u_t+\operatorname{div}\!\big(f[\mu^u_t](t,u_t,\cdot)\,\mu^u_t\big)=0\\
          \mu^u_{|t=0}=\mu_0
          \end{array}
          \right., \\
        &g(\mu^u(t)) \leq 0 \quad \forall t \in [0,T]
    \end{aligned}
\end{equation*}
where $u \equiv u_t: [0,T] \to \R^r$ is a control in open-loop form and $f$ may depend on the whole distribution $\mu_t$, and $\mu_0 \in \sP_2(\R^d)$ is an arbitrary measure.
Such problems arise naturally in the mean-field description and control of large populations. Typical examples include models of collective dynamics \cite{corbetta2016multiscale, cristiani2014multiscale, piccoli2009pedestrian}, opinion dynamics on networks \cite{bellomo2013dynamics, hegselmann2015opinion}, and animal flocks \cite{ballerini2008interaction, cucker2007emergent}. They also provide a natural framework for controlled systems whose initial condition is only known through a probability distribution \cite{coyaudOptimalTransportAppliedToPhysics}.
\smallskip

Pontryagin-type necessary conditions for optimal control problems in
Wasserstein spaces have received considerable attention in recent years.
For unconstrained problems, a Pontryagin maximum principle (PMP) was
established in \cite{bonnetrossi2019} and subsequently extended in several
directions; see, for instance, \cite{achdou2020mean,almi2025pontryagin}.
Formulations incorporating final-point and pointwise state constraints
have also been obtained
\cite{bonnet2019pontryagin,bonnet2021necessary,urrea2026pontryagin}.
However, these formulations rely on differentiability assumptions on
the cost and constraint functionals, which restrict their applicability
to natural objectives involving transport distances. In particular,
the squared Wasserstein distance $d_\mcW^2(\cdot,\sigma)$ is not globally
differentiable unless $\sigma$ is a Dirac mass
\cite{alfonsi2020squared}. Problems in which this functional measures
proximity to a prescribed distribution are therefore not generally
covered by these differentiable formulations.

This work introduces a first-order non-smooth calculus on $\sP_2(\R^d)$ that makes it possible to treat state-constrained optimal control problems with locally Lipschitz data. Building on this calculus, we establish a Pontryagin maximum principle beyond the differentiable framework covered by the existing literature. 
\smallskip

The contributions of this work are twofold. First, we introduce an
intrinsic Clarke-type subdifferential for locally Lipschitz functionals
on $\sP_2(\R^d)$, in close analogy with the classical Clarke
subdifferential over Banach spaces
\cite{clarke2013functional,clarke1998nonsmooth,clarke1975generalized,
clarke1976new,vinter2010optimal}.
To the best of our knowledge, this is the first use of a Clarke-type
subdifferential in a Wasserstein setting.
The choice of this object is motivated by the possible emptiness of
the Fr\'echet subdifferential of Ambrosio--Gigli--Savar\'e
\cite{ambrosioGradientFlowsMetric2005}, even for locally Lipschitz
functionals. For example, the Fr\'echet subdifferential of
$d_\mcW^2(\cdot,\sigma)$ at $\mu$ is empty unless the optimal transport
plan between $\mu$ and $\sigma$ is unique and induced by a map
\cite[Theorem~3.2]{alfonsi2020squared}.
It therefore cannot provide a generally applicable first-order
framework under the local-Lipschitz assumptions considered here.
By contrast, our Clarke subdifferential is non-empty, convex, and
weakly closed at every $\mu\in\sP_2(\R^d)$.
We establish its structural and stability properties and relate it
to the classical Clarke subdifferential through the Lions lift
\cite{lionslifting2006}.
\smallskip

This calculus also yields geometric characterizations of nonsmooth
functionals. Under convexity along plans, we show that the map-based
Clarke subdifferential satisfies a global supporting inequality against
arbitrary transport plans. We further obtain an explicit formula for
the Clarke subdifferential of the squared Wasserstein distance in terms
of barycentric projections of optimal transport plans, using a sharpened
characterization of its strong Fr\'echet superdifferential.
Together with the stability results, these characterizations provide
a framework for nonsmooth variational analysis and optimization on
Wasserstein space beyond the particular control problem considered here.
Additional proofs and geometric ingredients are collected in
Appendix~\ref{append: Additional properties Clarke subdiff}.
\smallskip

Second, we use this framework to prove a PMP for state-constrained
Bolza problems with locally Lipschitz terminal costs, running costs,
and state constraints, without assuming compact support of the initial
measure. The running cost satisfies the additional structural
assumptions stated in Section~\ref{subsec: setting of the problem}.
The proof relies on a regularization-and-stability argument:
we adapt the Lasry--Lions regularization \cite{lasry1986remark}
to $\sP_2(\R^d)$, regularize the cost and constraint functionals,
and derive necessary conditions for the resulting problems through
diffusive perturbations of the control
\cite{li2012optimal}.
To accommodate initial measures without compact support, we refine
the perturbation technique used in \cite{urrea2026pontryagin}
by discretizing the underlying space
$L^1([0,T];L^2_{\mu_0})$ rather than $\R^d$; see
Proposition~\ref{prop:diffuse-linearization}.
This extends the compactly supported framework of
\cite{bonnet2019pontryagin,bonnet2021necessary,bonnetrossi2019,
urrea2026pontryagin}.
Finally, to recover first-order information in the limit, we introduce
a pullback map based on Wasserstein parallel transport
\cite{bertucci2025tangent} and establish the corresponding stability
properties of the Clarke subdifferential.

The use of Hilbert-space tools in this argument does not make the result
a direct consequence of an abstract maximum principle.
Although the state equation admits an equivalent Lagrangian formulation
in $L^2_{\mu_0}$, differentiability of the velocity field in the spatial
and measure variables need not imply Fr\'echet differentiability of
$X\mapsto f[X_\#\mu_0](t,u,X(\cdot))$ as a map from $L^2_{\mu_0}$
to itself. Abstract results requiring this stronger regularity therefore
cannot be invoked under our assumptions alone; see also
\cite[Remark~4.3]{averboukh2025pontryagin} for a related discussion.
Our argument instead combines Hilbert-space regularization of the
scalar cost and constraint functionals with first-order expansions
along controlled perturbations of the flow. The intrinsic Clarke
calculus then provides the stability needed to identify limiting
subgradients as the underlying measures vary.

\paragraph{Structure of the paper.} In Section \ref{sec: formulation of the problem}, we introduce the optimal control problem and the main assumptions of the manuscript. Section \ref{sec: preliminaries on Wasserstein} recalls standard tools on $\sP_2(\R^d)$. In Section \ref{sec: Clarke subdiff on Wasserstein}, we introduce the Clarke subdifferential on $\sP_2(\R^d)$. Section \ref{sec: Lasry Lions on Wasserstein} is devoted to the study of the Lasry-Lions regularization on $\sP_2(\R^d)$. Section \ref{sec: PMP on Wasserstein} contains the proof of the main theorem and states the diffusive perturbation result. In Appendix \ref{append: Additional properties Clarke subdiff}, we defer the proofs of additional properties of the Clarke subdifferential which we consider to be of independent interest, while Appendix \ref{sec: diffusive perturbation} contains the proof of the diffusive perturbation and Appendix \ref{sec: well-posedness adjoint} provides a well-posedness result for the adjoint equation.



\section{Formulation of the problem}\label{sec: formulation of the problem}

\subsection{Preliminaries and notation}

Let $\mathcal{X}$ be a Polish space, i.e., a complete separable metric space, $\mcM(\mathcal{X})$ stands for the set of Radon measures over $\mathcal{X}$, while $\mcM_+(\mathcal{X})$ stands for the proper subset of $\mcM(\mathcal{X})$ of non-negative Radon measures. We say that a sequence of measures $\{\lambda_n\}_n \subset \mcM(\mathcal{X})$ {\it narrowly} converges or weakly-$\star$ converges to a measure $\lambda$ if $\lambda_n \rightharpoonup \lambda$ in duality with the set of continuous bounded functions $\mathcal{X} \to \R$.\\

$C(\mathcal{X})$ stands for the set of continuous functions $\mathcal{X} \to \R$. If $\mathcal{X}$ is compact, it is well known that the topological dual of $C(\mathcal{X})$ is $\mcM(\mathcal{X})$ and by Banach-Alaoglu theorem, bounded sets of $\mcM(\mathcal{X})$ are compact in the weak-$\star$ topology (i.e. in the topology induced by the narrow convergence).\\

We say that a function $\phi: [0,T] \to \mathcal{X}$ is a function of bounded variation if 
$$\operatorname{Var}(\phi) := \sup_{ \substack{ n \in \N \\ 0= t_1 < t_2 < ... < t_n = T} } \bigg\{ \sum_{k = 1}^{n-1} d(\phi(t_{k+1}), \phi(t_k)) \bigg\} < + \infty.$$

Let $(\mathcal{X},\Sigma_\mathcal{X})$ and $(\mathcal{Y},\Sigma_\mathcal{Y})$ be measurable spaces and let $\phi:\mathcal{X}\to \mathcal{Y}$ be measurable. The push-forward of a probability measure $\mu\in\sP(\mathcal{X})$ through $\phi$ is the probability measure $\phi_\#\mu\in\sP(\mathcal{Y})$ defined by
$$
\phi_\#\mu(A) := \mu(\phi^{-1}(A)), \qquad \forall A\in\Sigma_\mathcal{Y}.
$$

Let $(\mathcal{X},d)$ be a Polish metric space, i.e., a complete and separable metric space. The space $\sP_c(\mathcal{X})$ stands for the set of probability measures with compact support. We denote by $\sP_2(\mathcal{X})$ the space of probability measures with finite $2$-moment, namely

$$
\sP_2(\mathcal{X}) := \left\{ \mu\in\sP(\mathcal{X}) : \int_\mathcal{X} d^2(x,o)\,d\mu(x)<\infty \text{ for some  } o\in \mathcal{X} \right\}.
$$

The Wasserstein distance between two measures $\mu,\nu\in\sP_2(\mathcal{X})$ is defined by
$$
d_\mcW^2(\mu,\nu) := \inf_{\gamma\in\Gamma(\mu,\nu)} \int_{\mathcal{X}\times \mathcal{X}} d^2(x,y)\,d\gamma(x,y),
$$
where $\Gamma(\mu, \nu)$ denotes the set of transport plans between $\mu$ and $\nu$, say,
$$
\Gamma(\mu,\nu) := \left\{ \gamma\in\sP(\mathcal{X}\times \mathcal{X}) : \pi^1_\#\gamma=\mu,\ \pi^2_\#\gamma=\nu \right\}.
$$
Here $\pi^1,\pi^2: \mathcal{X}\times \mathcal{X}\to \mathcal{X}$ are the canonical projections onto the first and second components, respectively. 
We denote by $\Gamma_o(\mu,\nu)$ the set of optimal plans. In the special case of $\mathcal{X} = \R^d$, for $\mu \in \sP(\R^d)$ we define
$$M_2(\mu) := \bigg( \int_{\R^d} |x|^2 d \mu(x) \bigg)^{1/2}$$
so that $\mu \in \sP_2(\R^d)$ if and only if $M_2(\mu)  < +\infty$. We denote by $L^2_\mu$ the Lebesgue space of $\mu$-squared integrable mappings $\R^d \to \R^d$ with inner product $\inner{\cdot, \cdot}_\mu$ and induced norm $\norm{\cdot}_\mu$.\\

Finally, we say that a function $\mcG: \sP_2(\R^d) \to \R \cup \{ + \infty\}$ is {\it convex along plans} if for every $\mu, \nu \in \sP_2(\R^d)$ and any plan $\gamma \in \Gamma(\mu, \nu)$ it holds that
$$ \mcG\left( ( (1-t)\pi^1 + t \pi^2 )_\# \gamma \right) \leq (1-t) \mcG(\mu) + t \mcG(\nu), \qquad \forall t \in [0,1].$$

\subsection{Setting of the problem}\label{subsec: setting of the problem}

Let $T \in [0, + \infty)$ be a final horizon, $(U,d_U)$ a separable metric space and $\mu_0 \in \sP_2(\R^d)$ an initial condition. Let
$$f: [0,T] \times U \times \R^d \times \sP_2(\R^d) \to \R^d, \quad (t, u, x, \mu) \mapsto f[\mu](t,u,x) := f(t, u, x, \mu)$$
be a non-local dynamic. The space $\mcU_{ad}:= L^2((0,T);U)$ of admissible controls is defined as the space of all square-integrable mappings $[0;T] \to U$ modulo $\sL$-a.e. equality, endowed with the metric
$$d_{L^2}^2(u_1,u_2) := \int_0^T d_U^2(u_1(t), u_2(t)) dt, \qquad \forall u_1, u_2 \in \mcU_{ad}.$$

For a given $u \in \mcU_{ad}$, $\mu^u: [0,T] \to \sP_2(\R^d)$ stands for the (unique) solution of the non-local continuity equation
\begin{align}\tag{CE}\label{eq: continuity eq}
    \begin{cases}
        \partial_t\mu_t+\operatorname{div}\!\big(f[\mu_t](t,u_t,\cdot)\,\mu_t\big)=0 \qquad \forall t \in (0,T)\\
    \mu_{|t=0}=\mu_0,
    \end{cases}
\end{align}

where the solutions are understood in the distributional sense, i.e.,
$$\forall \vphi \in C_c^\infty((0,T) \times \R^d), \qquad \int_0^T \int_{\R^d} \partial_s \vphi(s,x) + \inner{\nabla_ x \vphi(s,x), f[\mu_s](s,u_s, x)} d \mu_s(x) dt = 0.$$

Under suitable assumptions on the dynamics (see Section \ref{sec: continuity eq}), for every $u \in \mcU_{ad}$ there is a unique curve $\mu^u_t \in AC([0,T]; \sP_2(\R^d))$ solving \eqref{eq: continuity eq}.\\

Let $\mcJ: \sP_2(\R^d) \to \R$ be a terminal-point cost, $L: [0,T] \times U \times \sP_2(\R^d) \to \R$ a Lagrangian and let $g: \sP_2(\R^d) \to \R$ be a constraint. We are interested in the optimal control problem
\begin{equation}\tag{OCP}\label{prob:P}
    \begin{aligned}
        \min_{u \in\mathcal{U}_{ad}} \quad J(u)&:= \mcJ(\mu^u(T)) + \int_0^T L(t, u(t), \mu^u(t))\,\mathrm{d}t \\
        \text{s.t} \quad &(\mu^u,u) \ \text{solves }   \left\{
          \begin{array}{l}
          \partial_t\mu_t+\operatorname{div}\!\big(f[\mu_t](t,u_t,\cdot)\,\mu_t\big)=0\\
          \mu_{|t=0}=\mu_0,
          \end{array}
          \right., \\
        &g(\mu^u(t)) \leq 0 \quad \forall t \in [0,T].
    \end{aligned}
\end{equation}

\begin{assu}[Dynamics]\label{hyp:dynamic-non-smooth}\textcolor{white}{.}
    \begin{enumerate}[i)]
      \item For every \((u,x,\mu)\), the map
      \(t\mapsto f[\mu](t,u,x)\) is measurable, and for a.e. \(t\in[0,T]\), the map
      \(
      (u,x,\mu)\mapsto f[\mu](t,u,x)
      \)
      is continuous.
    
      \item There exists \(\ell_f\in L^2(0,T)\) such that, for a.e.
      \(t\in[0,T]\), for every \(u,v\in U\), \(x,y\in\R^d\), and
      \(\mu,\nu\in\mathscr P_2(\R^d)\),
      \[
      \big|
      f[\mu](t,u,x)-f[\nu](t,v,y)
      \big|
      \le
      \ell_f(t)
      \big(
      d_U(u,v)+|x-y|+ d_\mcW(\mu,\nu)
      \big).
      \]
    
      \item There exist \(m_f\in L^1(0,T)\) and \(c_f>0\) such that, for
      a.e. \(t\in[0,T]\), for every \(u\in U\), \(x\in\R^d\), and
      \(\mu\in\mathscr P_2(\R^d)\),
      \[
      |f[\mu](t,u,x)|
      \le
      m_f(t)+c_f\big(d_U(u,o)+|x| + M_2(\mu)\big).
      \]

      \item For a.e. \(t\in[0,T]\), every \(u\in U\), and every
      \(\mu\in\mathscr P_2(\R^d)\), the map \(x\mapsto f[\mu](t,u,x)\) is
      continuously differentiable, and the map
      \[
      \mu\mapsto f[\mu](t,u,x)
      \]
      is differentiable \(\nabla_\mu f[\mu](t,u,x)(\cdot)\) in the sense of Definition \ref{def: differentiability on Wasserstein}. 

    
      \item For a.e. \(t\in[0,T]\) the derivatives \((u,x,\mu)\mapsto D_x f[\mu](t,u,x)\) and \((u,x,\mu,y) \mapsto \nabla_\mu f[\mu](t,u,x)(y)\) are continuous uniformly on bounded subsets. 
    \end{enumerate}
\end{assu}

\begin{assu}[Cost and constraint]\label{hyp:cost-non-smooth}
    \hfill
    
    The functions \(\mcJ\) and \(g\) are locally Lipschitz on \(\mathscr P_2(\R^d)\).
\end{assu}

\begin{assu}[Lagrangian]\label{hyp: assumptions lagrangian}
    \hfill
    
    \begin{enumerate}[i)]
        \item For every $\mu \in \sP_2(\R^d)$, the mapping $(t,u) \in [0,T] \times U \mapsto L(t,u, \mu)$ is continous.

        \item For every $\mu \in \sP_2(\R^d)$, there is $\epsilon > 0$ and $C >0$ such that for any $(t,u) \in [0,T] \times U$ it holds
        $$\forall \mu_1, \mu_2 \in B_\mcW[\mu, \epsilon], \qquad|L(t, u, \mu_2) - L(t, u, \mu_1)| \leq C d_\mcW(\mu_2, \mu_1).$$
          
        \item Either the mapping  \(\mu \mapsto L(t,u,\mu)\) is convex along plans for every \((t,u)\) or $L$ can be written as \(L(t,u,\mu) = L_1(t,u) + L_2(\mu)\).
    \end{enumerate}
    
\end{assu}

\subsection{Main result}

Define the Hamiltonian of the system \eqref{eq: continuity eq} as the mapping $H : [0,T] \times U \times\sP_2(T\R^d) \times \R \to \R$ given by
$$H(t,u, \gamma, \alpha) := \int_{T\R^d} \inner{r, f[\pi^1_\# \gamma](t,u,x)} d \gamma(x,r) - \alpha L(t,u, \pi^1_\# \gamma).$$

The main result of this manuscript is stated below. Here, $\partial_C$ stands for the Clarke subdifferential of Definition \ref{defi: clarke subdifferential}. The tangent space of $\R^d$ is denoted by $T\R^d \equiv \R^d \times \R^d$.

\begin{theorem}[Non-smooth PMP]\label{thm: non smooth pmp}
    Assume [A\ref{hyp:dynamic-non-smooth}]-[A\ref{hyp: assumptions lagrangian}]. Let \((\mu^*,u^*)\) be a local minimizer of
  \eqref{prob:P}. Then there exist
  \(
  \alpha^*\in[0,1]\),
  \(
  \lambda^*\in\mathcal M_+([0,T]),
  \)
  a co-state $\gamma^* \in BV \left([0,T],\mathscr{P}_2( T \R^d) \right)$, measurable selections
  \[
  \zeta_\mcJ^*\in\partial_C\mcJ(\mu_T^*), \quad
  \zeta_L^*(t)\in
  \partial_C L(t,u_t^*,\cdot)(\mu_t^*)
  \qquad\text{for a.e. }t\in[0,T],
  \]
  and 
  \(
  \zeta_g^*(t)\in
  \partial_C g(\mu_t^*)\)
 \(\lambda^*\text{-a.e. }t\in[0,T]\),
  such that the following properties hold.
    
    \begin{enumerate}
    \item \textbf{Non-triviality.}
    \(
    (\alpha^*,\lambda^*)\neq(0,0)
    \).
    
    \item \textbf{Complementary slackness.}
    \(
    \operatorname{supp}(\lambda^*)
    \subset
    \big\{
    t\in[0,T]\;:\; g(\mu_t^*)=0
    \big\}
    \).
    
    \item \textbf{Adjoint equation:}
    The co-state measure is given by
    \[
    \gamma_t^*
    =
    \big(\Phi^*_{(0,t)}(\cdot),P^*(t,\cdot)\big)_\#\mu_0,
    \qquad \forall t\in[0,T]
    \]
    where, for
    \[
    m^*(s,x)
    :=
    \int_{[0,s]}
    \zeta_g^*(t)\big(\Phi^*_{(0,t)}(x)\big)\,
    \lambda^*(dt),
    \qquad s\in[0,T],
    \]
    the mapping \(P^*\in BV\left([0,T];L^2_{\mu_0} \right) \) is the unique solution of
    \[
    \left\{
    \begin{aligned}
     dP^*(t,x)
    &=
    \bigg[- 
    D_x f[\mu_t^*]\big(t,u_t^*,\Phi^*_{(0,t)}(x)\big)^\top P^*(t,x)\\
    &\quad
    -
    \int_{\R^d}
    \nabla_\mu f[\mu_t^*]\big(t,u_t^*,\Phi^*_{(0,t)}(y)\big)
    \big(\Phi^*_{(0,t)}(x)\big)^\top P^*(t,y)\,d\mu_0(y)\\
    &\quad
    +
    \alpha^*
    \zeta_L^*(t)\big(\Phi^*_{(0,t)}(x)\big) \bigg]dt + dm^*(t,x)\\
    P^*(T,x)
    &=
    -\alpha^*
    \zeta_\mcJ^*\big(\Phi^*_{(0,T)}(x)\big)
    \end{aligned}
    \right.
    \]
    for $\mu_0$-almost every $x$. In particular,
    \(
    \pi^1_\#\gamma_t^*=\mu_t^*
    \).
    
    \item \textbf{Maximization condition:}
    For a.e. \(t\in[0,T]\),
    \[
    H(t,u_t^*,\gamma_t^*,\alpha^*)
    =
    \max_{u\in U}
    H(t,u,\gamma_t^*,\alpha^*).
    \]
    \end{enumerate}
\end{theorem}
    \begin{remark}[Several state constraints]
    Theorem~\ref{thm: non smooth pmp} extends directly to a finite family of state
    constraints
    \[
    \psi_i(\mu_t)\le 0,
    \qquad i=1,\ldots,m.
    \]
    Indeed, setting
    \(
    \mu\mapsto\psi(\mu):=\max_{1\le i\le m}\psi_i(\mu),
    \)
    the family of constraints is equivalent to the single condition
    \[
    \psi(\mu_t)\le0.
    \]
    If each $\psi_i$ is globally Lipschitz, then so is $\psi$, and
    Theorem~\ref{thm: non smooth pmp} applies without any modification of its proof.
    Moreover, by Theorem~\ref{thm: computational properties Clarke subdiff},
    \[
    \partial_C\psi(\mu)
    \subset
    \operatorname{co}
    \bigcup_{i\in I(\mu)}
    \partial_C\psi_i(\mu),
    \qquad
    I(\mu):=
    \left\{
    i:\psi_i(\mu)=\max_j\psi_j(\mu)
    \right\}.
    \]
    Consequently, the constraint selection appearing in the adjoint equation
    satisfies, for $\lambda^*$-a.e. $t$,
    \[
    \zeta_\psi^*(t)
    \in
    \operatorname{co}
    \bigcup_{i\in I(\mu_t^*)}
    \partial_C\psi_i(\mu_t^*).
    \]
    In particular, on $\operatorname{supp}(\lambda^*)$, complementary slackness
    implies
    \[
    I(\mu_t^*)
    =
    \left\{
    i:\psi_i(\mu_t^*)=0
    \right\},
    \]
    so the measure contribution to the adjoint involves only subgradients of
    active constraints.
    \end{remark}
\begin{remark}[On the assumptions on the running cost]\label{rem:convexity-running-cost}
    As already mentioned, the proof of the above theorem relies on a regularization-and-stability argument. In the passage to the limit, one has to identify weak limits of gradients of the form 
    \[
    \nabla_\mu L^\epsilon
    \bigl(t^\epsilon,u^\epsilon,\mu^\epsilon\bigr), \qquad     (t^\epsilon,u^\epsilon,\mu^\epsilon)
    \longrightarrow
    (t,u,\mu).
    \]
    The main difficulty is that the stability results for the Lasry-Lions approximation (cf. Propositions \ref{prop:closed-graph-clarke-det} and \ref{prop: Clarke subdiff and Lasry Lions}) apply to the functional $L(t,u,\cdot)$ for fixed $t,u$, whereas $(t^\epsilon, u^\epsilon)$ varies as $\epsilon \searrow 0$. Under convexity along plans, Theorem \ref{thm: clarke subdiff and convex mappings} characterizes elements of the Clarke subdifferential through a global inequality. This characterization provides the additional stability needed to conclude that every pullback limit belongs to $\partial_C L(t,u,\cdot)(\mu)$, see Proposition~\ref{prop:running-cost-LL-closedness}.\\
    
    This convexity along plans can be avoided when the running cost has the separated
    form
    \[
    L(t,u,\mu)=L_1(t,u)+L_2(\mu).
    \]
    In this case,
    \[
    L^\epsilon(t,u,\mu)
    =
    L_1(t,u)+L_2^\epsilon(\mu),
    \]
    so its Wasserstein gradient is independent of \(t\) and \(u\). The
    identification of its weak limits then follows from the closedness result
    for the fixed functional \(L_2\). Consequently, the conclusion of
    Theorem~\ref{thm: non smooth pmp} remains valid, with
    \[
    \zeta_L^*(t)
    \in
    \partial_C L_2(\mu_t^*)
    =
    \partial_C L(t,u_t^*,\cdot)(\mu_t^*)
    \qquad\text{for a.e. }t\in[0,T].
    \]
    If $L$ does not satisfy assumption [A\ref{hyp: assumptions lagrangian}.iii],
    it is not clear whether the regularization-and-stability argument can be
    carried out.
\end{remark}
    
\begin{remark}[On the assumptions on the dynamics]
    The non-smoothness considered in this manuscript concerns the cost and constraint functionals, while the dynamics retain the regularity needed to perform a first-order analysis of the state equation.\\
    
    The Lipschitz and growth assumptions on \(f\) ensure the well-posedness of the continuity equation and its Lagrangian representation
    \[
    \mu_t^u=\bigl(\Phi^u_{(0,t)}\bigr)_\#\mu_0.
    \]
    They also provide the stability of the state with respect to the control, which is used throughout the variational and limiting arguments. In particular, perturbations of the control produce controlled perturbations of both the trajectory and the associated flow.\\ 
    
    The differentiability of \(f\) with respect to the state and measure variables allows one to obtain a first-order expansion of the flow. Together with the continuity of these derivatives, it yields the well-posedness and stability of the linearized state equation, whose dual equation gives the adjoint system appearing in the maximum principle.\\
    
    Removing the differentiability assumptions on the dynamics would make the analysis substantially more involved, since the linearized flow and the corresponding dual construction would no longer be directly available. It also remains unclear whether the graph-concentration property of the adjoint trajectory $\gamma_t^*$ is a consequence of the differentiability of $f$, or whether it persists for non-smooth dynamics.
\end{remark}

\section{Preliminaries on $\sP_2(\R^d)$}\label{sec: preliminaries on Wasserstein}

In this section we recall different constructions and tools on $\sP_2(\R^d)$. 

\subsection{Tangent space}
\hfill

Let $\mu \in \sP_2(\R^d)$. We define the space $\sP_2(T\R^d)_\mu$ as the set
$$\sP_2(T\R^d)_\mu := \{ \xi \in \sP_2(T\R^d) : \pi^1_\# \xi = \mu \} \subset \sP_2(T\R^d).$$

By the disintegration theorem \cite[Theorem 5.3.1]{ambrosioGradientFlowsMetric2005}, any measure $\xi \in \sP_2(T\R^d)_\mu$ can be written as $\xi(dx, dv) = \mu(dx) \xi_x(dv)$ where $\xi_x \in \sP_2(\R^d)$. The {\it barycentric} projection $\mcB: \sP_2(T\R^d)_\mu \to L^2_\mu$ is the mapping
\begin{align}\label{eq: definition barycentric projection}
    \xi \in \sP_2(T\R^d)_\mu \mapsto \bigg( x \mapsto \mcB[\xi](x) := \int v  d \xi_x(v) \bigg) \in L^2_\mu.
\end{align}

The {\it regular} tangent space of $\sP_2(\R^d)$ at $\mu$, denoted by $\Tanreg_\mu \sP_2(\R^d)$ is the set of functions
$$\Tanreg_\mu \sP_2(\R^d) := \overline{\nabla C_c^\infty(\R^d)}^{L^2_\mu} \subset L^2_\mu.$$
It is known \cite{ambrosioGradientFlowsMetric2005} that $\Tanreg_\mu \sP_2(\R^d)$ corresponds to the set of limit points of mappings $T: \R^d \to \R^d$ that are optimal in small time, i.e., of the mappings for which $ (\id, \id + t T)_\# \mu \in \Gamma_o(\mu, (\id + tT)_\# \mu)$ for $t$ small. 

There is a natural injection $\iota_\mu: L^2_\mu \hookrightarrow \sP_2(T\R^d)_\mu$ given by $\psi \mapsto \iota_\mu(\psi) :=(\id, \psi)_\# \mu$. By definition, it holds that $\mcB \circ \iota_\mu = \id$ in $L^2_\mu$.

\begin{remark}
    Although $\sP_2(T\R^d)_\mu$ is not a tangent space in the metric-geometric sense, the variations arising in the present control problem are generally not geodesic. This makes $\sP_2(T\R^d)_\mu$ the appropriate space to work on.
\end{remark}

\subsection{Lifting technique}

The lifting technique on $\sP_2(\mathbb{R}^d)$ was introduced by P.L. Lions in \cite{lionslifting2006} as a tool for the study of Hamilton-Jacobi equations on this space. Since then, it has found applications well beyond this original setting. In this manuscript, we  use this technique many times; see, e.g., Sections~\ref{sec: Clarke subdiff on Wasserstein} and \ref{sec: Lasry Lions on Wasserstein}.\\

From now on,  we fix a complete non-atomic probability space $(\Omega; \mbP)$ where $\Omega$ is a Polish metric space. We denote by $L^2_\mbP$ the space of random variables $X: \Omega \to \R^d$ that are squared-integrable, and we often write $X \sim \mu$ if $X_\# \mbP = \mu$. It is well-known that the mapping $X \in L^2_\mbP \mapsto X_\# \mbP \in \sP_2(\R^d)$ is surjective, see, e.g., \cite[Lemma 5.29]{carmonaProbabilisticTheoryMean2018}.\\

On the one hand, any function $\mcG: \sP_2(\R^d) \to \R$ can be lifted to a function $\hat{\mcG}: L^2_\mbP \to \R$ by $X \mapsto \hat{\mcG}(X) := \mcG(X_\# \mbP)$. On the other hand, any function $\hat{\mcG}: L^2_\mbP\to \R$ that is law-invariant, meaning that $\hat{\mcG}(X)$ only depends on $X_\# \mbP$, induces a map $\mcG: \sP_2(\R^d)\to \R$ through
$$\mcG(\mu) := \hat{\mcG}(X) \quad \text{ with } X_\# \mbP = \mu.$$
By definition, $\hat{\mcG}$ is the actual lift of $\mcG$. A direct application of \cite[(5.14) p. 358]{carmonaProbabilisticTheoryMean2018} yields the following.

\begin{proposition}
    A function $\mcG$ is Lipschitz continuous if and only if $\hat{\mcG}$ is. Similarly, if $\hat{\mcG}: L^2_\mbP \to \R$ is Lipschitz continuous and law-invariant, then its induced mapping $\mcG: \sP_2(\R^d) \to \R$ is Lipschitz continuous as well. In any case, the Lipschitz constant for $\mcG$ and $\hat{\mcG}$ is the same.
\end{proposition}

\subsection{Differential calculus}

The study of differentials and semi-differentials on $\sP_2(\R^d)$ is a broad subject. An intrinsic notion of semi-differentials on $\sP_2(\R^d)$ was first introduced in \cite{ambrosioGradientFlowsMetric2005}, while an extrinsic definition of differentiability was exposed in \cite{lionslifting2006} for the study of Hamilton-Jacobi equations and further studied in \cite{carmonaProbabilisticTheoryMean2018}; see also \cite{cardaliaguetNotesMeanField}. In \cite{gangboDifferentiabilityWassersteinSpace2019} the connection between these two points of view is made.

\begin{definition}[Differentiable functions]\label{def: differentiability on Wasserstein}
    Let $\mcG: \sP_2(\R^d) \to \R$ be a function. We say that $\mcG$ is differentiable at $\mu \in \sP_2(\R^d)$ if there is a mapping $\xi \in L^2_\mu$ such that for every $\gamma \in \sP_2(\R^{2d})_{\mu}$ it holds that
    \begin{align}\label{eq: diff on Wasserstein}
        \mcG(\pi^2_\# \gamma) = \mcG(\mu) + \int \inner{\xi(x), y-x} d \gamma(x,y) + o\left(\norm{\pi^2 - \pi^1}_\gamma \right).
    \end{align}
    In such a case, the mapping $\xi$ is unique in $\Tanreg_\mu \sP_2(\R^d)$. We denote it by $\nabla_\mu \mcG(\mu)$.
\end{definition}

\begin{definition}[$C^{1,1}$ mappings]\label{def: C11 on Wass}
    We say that a mapping $\mcG: \sP_2(\R^d) \to \R$ is of class $C^{1,1}$ if its lift $\hat{\mcG}$ is of class $C^{1,1}$ on $L^2_\mbP$. Equivalently, $\mcG$ is of class $C^{1,1}$ if it is differentiable in the sense of Definition \ref{def: differentiability on Wasserstein} and there is $L_\mcG > 0$ such that
    $$\int_{\R^{2d}} | \nabla_\mu \mcG(\nu) (y) - \nabla_\mu \mcG(\mu) (x)|^2 d \gamma(x,y) \leq L_\mcG^2 \int_{\R^{2d}} |y-x|^2 d \gamma(x,y), \qquad \forall \gamma \in \Gamma(\mu,\nu).$$
\end{definition}

The following theorem is due to \cite{gangboDifferentiabilityWassersteinSpace2019}. We point out that the structure of the differential of a law-invariant mapping was already known by P.L. Lions in the case of $C^{1,1}$ mappings, see \cite{lionslifting2006}, \cite[Proposition 5.24 and 5.25]{carmonaProbabilisticTheoryMean2018} or \cite[Theorem 6.5]{cardaliaguetNotesMeanField}.

\begin{theorem}\label{thm:gangbo-tudorascu}
 A mapping $\mcG: \sP_2(\R^d) \to \R$ is differentiable at $\mu$ if and only if its lifted function $\hat{\mcG}$ is differentiable at every random variable $X \sim \mu$. In this case, it holds that $\nabla \hat{\mcG}(X) = \nabla_\mu \mcG \circ X$ for every $X \sim \mu$.
\end{theorem}

\subsection{Parallel transport}

We previously saw that the gradient of a function $\mcG: \sP_2(\R^d) \to \R$ at a measure $\mu$ is an element belonging to $L^2_\mu$. This generates a problem when one needs to compare the gradients of a function at two different measures. This is why having a notion  of {\it parallel transport} on $\sP_2(\R^d)$ becomes necessary. Here we present its definition. This subsection is based in \cite{bertucci2025tangent}.

\begin{definition}[Parallel transport]\label{def:PT}
  Let $\mu,\nu\in\Ptwo$, $\gamma\in\Gamma(\mu,\nu)$ and $\eta \in \sP_2(T\R^d)_\nu$. Let $\eta(dy,dz)=\nu(dy)\psi_y(dz)$ be the disintegration of $\eta$. The \emph{parallel transport along $\gamma$} of $\eta$ is the element $\PT_\gamma(\eta) \in \sP_2(T\R^d)_\mu$ defined by
  \[
  \mathrm{PT}_\gamma(\eta)(dx,dz)
  :=\int_{\R^d}\psi_y(dz)\,\gamma(dx,dy).
  \]
  Equivalently, disintegrating $\gamma(dx,dy)=\mu(dx)\gamma_x(dy)$,
  \[
  \mathrm{PT}_\gamma(\eta)(dx,dz)=\mu(dx)\,\widetilde\psi_x(dz),
  \qquad
  \widetilde\psi_x:=\int_{\R^d}\psi_y\,\gamma_x(dy).
  \]

\end{definition}

One might think that the parallel transport of a mapping $\zeta \in L^2_\nu$ remains to be a mapping if we parallel transport it to a measure $\nu$. However this is not the case. In fact, if $\nu$ is concentrated on two points and $\mu$ is a Dirac mass, parallel transporting a mapping $\zeta \in L^2(\nu)$ from $\nu$ to $\mu$, where $\zeta(x) \neq 0$ for $\nu$-a.e. $x$, yields a measure on $\sP_2(T\R^d)_\mu$ that is not concentrated on a graph. To circumvent this problem, we introduce the {\it pullback} mapping.

\begin{definition}[Pullback mapping]\label{def:pullback-coupling}
    Let \(\mu,\nu\in\Ptwo\) and \(\gamma\in\Gamma(\mu,\nu)\). The \emph{pullback} mapping $\PP_\gamma: L^2_\nu \to L^2_\mu$ is defined by
    \[
    \beta\mapsto \mathsf P_\gamma \beta
    :=
    \mcB\big[\PT_\gamma(\iota_\nu(\beta)) \big]
    \in L^2_\mu.
    \]
    Equivalently, if $\gamma(dx,dy)=\mu(dx)\gamma_x(dy),$
    then
    \[ \mathsf P_\gamma \beta(x) = \int_{\R^d} \beta(y)\,\gamma_x(dy) \qquad \mu\text{-a.e. }x. \]
    In particular, by Jensen's inequality it holds that
    \begin{align}\label{ineq: bounds-pullback}
        \forall \beta \in L^2_\nu, \quad \forall \gamma \in \Gamma(\mu, \nu), \qquad \norm{ \mathsf{P}_\gamma \beta }_\mu \leq  \norm{\beta}_\nu.
    \end{align}
    
\end{definition}

\subsection{Continuity equation}\label{sec: continuity eq}

We recall some well-posedness and stability properties of the controlled
continuity equation~\eqref{eq: continuity eq} (see
\cite[Theorem~2]{bonnet2021differential} and
\cite[Theorem~2.15]{bonnet2019pontryagin}). Under
Assumption~\ref{hyp:dynamic-non-smooth}, for every
\(u\in\mcU_{ad}\), there exists a
unique solution $\mu^u\in\AC\bigl([0,T];\sP_2(\R^d)\bigr)$ of~\eqref{eq: continuity eq}. The flow associated with the vector field
\(
(t,x)\longmapsto f[\mu_t^u](t,u_t,x)
\)
is denoted by \(\Phi^u_{(0,t)}\) and satisfies
\begin{equation}\label{eq:def-flow}
\Phi^u_{(0,t)}(x)
=
x+
\int_0^t
f[\mu_s^u]
\bigl(s,u_s,\Phi^u_{(0,s)}(x)\bigr)\,ds.
\end{equation}
For every \(t\in[0,T]\), the map \(\Phi^u_{(0,t)}\) is invertible, and
its inverse is denoted by \(\Phi^u_{(t,0)}\). The solution admits the
representation
\begin{equation}\label{eq:flow-representation}
\mu_t^u
=
\bigl(\Phi^u_{(0,t)}\bigr)_\#\mu_0,
\qquad t\in[0,T].
\end{equation}

Let \(\mu^u\) and \(\nu^v\) be the solutions associated with controls
\(u,v\in\mcU_{ad}\) and initial conditions \(\mu_0,\nu_0\in\sP_2(\R^d)\), respectively. By
\cite[Proposition~4.4]{urrea2026pontryagin}, there exists a constant
\(C>0\), depending only on \(T\) and the constants in
[A\ref{hyp:dynamic-non-smooth}], such that
\begin{equation}\label{eq:stability-continuity-equation}
\sup_{t\in[0,T]}
d_\mcW(\mu_t^u,\nu_t^v)
\leq
C\left(
d_{L^2}(u,v)+d_\mcW(\mu_0,\nu_0)
\right)
\end{equation}
For a fixed initial condition \(\mu_0\in\sP_2(\R^d)\), the corresponding flows satisfy
\begin{align}\label{eq: stability flow}
    \left\|\Phi^u_{(0,\cdot)}-\Phi^v_{(0,\cdot)} \right\|_{C([0,T];L^2_{\mu_0})}
    \leq
C\,d_{L^2}(u,v).
\end{align}

In particular, if \(u_n\to u\) in \(\mcU_{ad}\), then
\(
\Phi^{u_n}_{(0,\cdot)}
\longrightarrow
\Phi^u_{(0,\cdot)}\)in \(
C\bigl([0,T];L^2_{\mu_0}\bigr)
\).

\begin{remark}
    We stress that in \cite[Proposition~4.4]{urrea2026pontryagin} the statements are announced for measures $\mu_0, \nu_0 \in \sP_c(\R^d)$. However, the same proof remains valid if $\mu_0$ and $\nu_0$ do not have compact support.
\end{remark}






\section{Clarke subdifferential on $\mathscr{P}_2(\R^d)$}\label{sec: Clarke subdiff on Wasserstein}

In this section we introduce a new notion of subdifferential on $\P_2(\R^d)$ which we call the Clarke subdifferential. For a given Hilbert space $H$ and a set $C \subset H$, $\sigma_C: H \to \R$ stands for the support function of $C$ defined as
$$\sigma_C(v) := \sup_{ x \in C} \inner{v,x}.$$

\subsection{Definition and main properties}

We begin by introducing a notion of generalized derivative. For this, we denote by $C_1(\R^d;\R^d)$ the space of continuous functions with at most linear growth, say,
$$C_1(\R^d;\R^d) := \bigg\{ w \in C(\R^d; \R^d) : \quad \sup_{x \in \R^d} \frac{|w(x)|}{1 + |x|} < + \infty \bigg\}.$$ 

\begin{definition}[Generalized derivative]
    Let $\mcG: \sP_2(\R^d) \to \R$ be locally Lipschitz around $\mu \in \sP_2(\R^d)$. The generalized derivative of $\mcG$ at $\mu$ in the direction $w \in C_1(\R^d;\R^d)$ is defined as the mapping
    $$D \mcG(\mu)(w) := \limsup_{\substack{\nu \to \mu\\\ t \searrow 0}} \frac{ \mcG( (\id + t w)_\# \nu) - \mcG(\nu)}{t}.$$
\end{definition}
It can be proven (see Proposition \ref{prop:basic-properties-clarke-subdiff} below) that $D \mcG(\mu)$ admits a unique Lipschitz extension $L^2_\mu \to \R$ which we still denote by $D \mcG(\mu)$.

\begin{definition}[Clarke subdifferential]\label{defi: clarke subdifferential}
    The Clarke subdifferential $\partial_C \mcG(\mu)$ of $\mcG$ at $\mu$ is defined as the set
    \[
    \begin{aligned}
    \partial_C \mcG(\mu)
    :=\left\{\Xi \in L_{\mu}^{2}:\forall w \in C_1(\R^d;\R^d), \quad
    \langle\Xi, w\rangle_{\mu} \leq D \mcG(\mu)(w)\right\} .
    \end{aligned}
    \]
\end{definition}

The squared Wasserstein distance is the basic example motivating the present non-smooth framework, since it is locally Lipschitz but generally fails to be smooth. The following theorem nevertheless provides a complete formula for its Clarke-subdifferential.

\begin{theorem}[Clarke subdifferential of the squared Wasserstein distance]\label{thm: clarke subdiff of wass dist}
    Let $\mu, \sigma \in \sP_2(\R^d)$. Then,
    $$\partial_C d_\mcW^2(\cdot,\sigma)(\mu) = \overline{\operatorname{conv}\bigg( \left\{ 2 \id- 2\mcB[\gamma] : \gamma \in \Gamma_o(\mu, \sigma) \right\} \bigg)}^{L^2_\mu}.$$
\end{theorem}

When there is a unique optimal plan and it is induced by a map $T$ (e.g., if $\mu$ is absolutely continuous w.r.t. the Lebesgue measure or if $\sigma$ is a Dirac mass), the above formula
reduces to
\[
\partial_C d_\mcW^2(\cdot,\sigma)(\mu)
=
\{2(\id-T)\}.
\]

Thus, the Clarke subdifferential extends the familiar Wasserstein
gradient to the non-smooth regime.\\

The proof combines two geometric refinements: a closure property for convex hulls in the space of tangent plans, and a sharpened characterization of the strong Fréchet superdifferential of
$d_\mcW^2(\cdot,\sigma)$. We defer these arguments to Appendix \ref{sec: clarke subdiff of squared wass dist}.

\begin{proposition}\label{prop:basic-properties-clarke-subdiff}
    Let $\mcG: \sP_2(\R^d) \to \R$ be a locally Lipschitz function around $\mu \in \sP_2(\R^d)$.
    \begin{enumerate}[i)]
        \item If $\Lip_\mu(\mcG)$ is a Lipschitz constant of $\mcG$ in a neighborhood of $\mu$, then for every $w \in C_1(\R^d;\R^d)$ it holds that
        $$|D \mcG(\mu)(w)| \leq \Lip_\mu(\mcG) \norm{w}_\mu.$$
        
        \item For every $w_1, w_2 \in C_1(\R^d;\R^d)$,
        $$|D\mcG(\mu)(w_2) - D\mcG(\mu)(w_1)| \leq \Lip_\mu(\mcG)  \norm{w_2 - w_1}_\mu.$$
        
        \item $D \mcG(\mu)$ admits a unique continuous extension to $L^2_\mu$ which is positively homogeneous and $\Lip_\mu(\mcG)$-Lipschitz continuous.
        
        \item For any $w \in C_1(\R^d;\R^d)$, the mapping $\mu \in \sP_2(\R^d) \mapsto D \mcG(\mu)(w)$ is upper semicontinuous in the topology of $\sP_2(\R^d)$.
        
        \item $\partial_C \mcG(\mu)$ is nonempty, convex and weakly closed.
        \item If $\Lip_\mu(\mcG)$ is a Lipschitz constant of $\mcG$ in a neighborhood of $\mu$, it holds that
        $$\forall \Xi \in \partial_C \mcG(\mu), \qquad \norm{\Xi}_\mu \leq \Lip_\mu(\mcG).$$
        \item For every $\beta \in L^2_\mu$ it holds that
        $$D \mcG(\mu)(\beta) = \sigma_{\partial_C \mcG(\mu)}(\beta).$$
    \end{enumerate}

\end{proposition}
    
\begin{proof}
    Statements i), ii) and iv) follow by a direct adaptation of \cite[Proposition 10.2]{clarke2013functional}, while iii) is a consequence of ii) and the density of $C_1(\R^d;\R^d)$ over $L^2_\mu$.\\

    To prove v), an application of Hahn-Banach theorem shows that the set $\partial_C \mcG(\mu)$ is nonempty. Closedness and convexity are direct consequences of the fact that
        \[ \partial_C  \mcG(\mu)
        =
        \bigcap_{v\in L^2_\mu}
        \left\{
        \Xi\in L^2_\mu:
        \langle \Xi,v\rangle_\mu\le D \mcG(\mu)(v)
        \right\}.
        \]

     To prove vi), let $\Xi \in \partial_C \mcG(\mu)$. Using the definition of Clarke subdifferential together with i),
        \[ \|\Xi\|^2_\mu =  \langle \Xi,\Xi\rangle_\mu \le D\mcG(\mu)(\Xi) \le |D \mcG (\mu)(\Xi)|
        \le \Lip_\mu(\mcG) \|\Xi\|_\mu. \]

    Finally, vii) follows directly from the fact that a positively homogeneous sublinear and continuous function can be represented as the support function of its convex subdifferential at $0$, see \cite[Theorem 4.25]{clarke2013functional}.
    
\end{proof}

\begin{proposition}\label{prop:closed-graph-clarke-det}
    Let \( \mcG :\P_2(\R^d)\to\R\) be locally Lipschitz around
    \(\mu\in\P_2(\R^d)\). Let $\mu_n \to \mu$, $\Xi_n \in \partial_C \mcG(\mu_n)$, \(\gamma_n\in\Gamma(\mu,\mu_n)\) be such that
    $$\norm{\pi^2 - \pi^1}_{\gamma_n}\to 0$$
    and let $\Xi \in L^2_\mu$ be a limit point of the sequence $\{\PP_{\gamma_n} \Xi_n\}_n$ in the weak topology of $L^2_\mu$. Then, $\Xi \in \partial_C \mcG(\mu)$.
\end{proposition}

\begin{proof}

    Let $\Lip_\mu(\mcG)$ be a Lipschitz constant of $\mcG$ in a neighborhood of $\mu \in \sP_2(\R^d)$. By Proposition~\ref{prop:basic-properties-clarke-subdiff},
    $$\|\Xi_n\|_{\mu_n} \le \Lip_\mu(\mcG)$$
    for $n \in \N$ sufficiently large. Let \( \psi\in C_c^\infty(\R^d;\R^d)\). Up to relabeling we write $\PP_{\gamma_n} \Xi_n \rightharpoonup \Xi$. Using the definition of the pullback mapping we have
    \[
    \begin{aligned}
    \left|
    \langle \PP_{\gamma_n}\Xi_n, \psi \rangle_\mu
    -
    \langle \Xi_n, \psi \rangle_{\mu_n}
    \right|
    &=
    \left|
    \int_{\R^{2d}}
    \langle \Xi_n(y),\psi (x)- \psi(y)\rangle
    \,d\gamma_n(x,y)
    \right|  \\
    &\le
    \|\Xi_n\|_{\mu_n}
    \left(
    \int_{\R^{2d}}| \psi(x)- \psi(y)|^2\,d\gamma_n(x,y)
    \right)^{1/2} \\
    &\le
    \Lip_\mu(\mcG) \,\Lip( \psi ) \, \norm{\pi^2 - \pi^1}_{\gamma_n}
    \longrightarrow0.
    \end{aligned}
    \]
    Since \(\Xi_n\in\partial_C \mcG(\mu_n)\), Proposition \ref{prop:basic-properties-clarke-subdiff} and the weak convergence $\PP_{\gamma_n} \Xi_n \rightharpoonup \Xi$ yield
    $$ \inner{\Xi, \psi}_\mu = \lim_{n \to + \infty}
    \langle \Xi_n, \psi \rangle_{\mu_n}
    \le \limsup_n D \mcG (\mu_n)( \psi) \leq D\mcG(\mu)( \psi ),$$
    which proves the statement.
\end{proof}

\subsection{Characterization through the Lions lift}

The preceding subsection introduced the Clarke subdifferential
intrinsically on $\sP_2(\R^d)$. We now show that it admits an equivalent characterization through the classical Clarke subdifferential of the lifted functional on $L^2_\mbP$.\\

We stress that the intrinsic formulation is particularly suitable for proving
properties directly on Wasserstein space, such as the closed-graph property in Proposition \ref{prop:closed-graph-clarke-det}. By contrast, the lifted formulation is more convenient for the regularization arguments developed in the next section.\\

For a given $U: L^2_\mbP \to \R$ locally Lipschitz, the generalized derivative $DU(X)(V)$ of $U$ at $X \in L^2_\mbP$ in the direction $V \in L^2_\mbP$ is defined as
$$DU(X)(V) := \limsup_{\substack{t \searrow 0\\ Y \to X}} \frac{U(Y + tV) - U(Y)}{t}$$

The Clarke subdifferential $\partial_CU(X)$ of $U$ at $X$ is the set
$$\partial_C U(X) := \left\{Z \in L^2_\mbP : \forall V \in L^2_\mbP, \quad \inner{Z,V}_\mbP \leq DU(X)(V) \right \}.$$

By \cite[Definition 10.3]{clarke2013functional} it holds that
\begin{align}\label{eq: generalized derivative and support function classical case}
    \forall V \in L^2_\mbP, \qquad DU(X)(V) = \sigma_{\partial_CU(X)}(V).
\end{align}

Given $X\in L^2_\mbP$, its barycentric projection $\Bar_X[\cdot]$ induced by $X$ is the mapping
$$Z \in L^2_\mbP \mapsto \Bar_X[Z] := \mcB[ (X,Z)_\#\mbP] \in L^2_{X_\# \mbP}$$
where $\mcB$ is the barycentric projection \eqref{eq: definition barycentric projection}. It can be proven that $\Bar_X[Z]$ is characterized by the following equality
$$\inner{\Bar_X[Z], \psi}_{L^2_\mu} = \inner{Z, \psi\circ X}_\mbP \qquad \forall \psi \in C_c^\infty(\R^d;\R^d).$$
Consequently, the mapping $\Bar_X[\cdot] : L^2_\mbP \to L^2_{X_\# \mbP}$ is linear and continuous.

\begin{proposition}[Characterization of Clarke subdifferential through the lift]\label{prop: clarke subdifferential and lifting}
    Let $\mcG: \sP_2(\R^d) \to \R$ be locally Lipschitz around $\mu$ and let $\hat{\mcG}: L^2_\mbP \to \R$ be its lift. Let $X \in L^2_\mbP$ be such that $X_\#\mbP = \mu$. Then,
    \begin{align}\label{eq: equality generalized derivative}
        \forall \beta \in L^2_\mu, \qquad D \hat{\mcG}(X)(\beta \circ X) =   D \mcG(\mu)(\beta)
    \end{align}
    Moreover,
    $$\partial_C \mcG(\mu) = \Bar_X[\partial_C \hat{\mcG}(X)].$$
\end{proposition}

\begin{proof}
    Let $X \in L^2_\mbP$ with $X_\# \mbP = \mu$ and let $\Lip_X(\hat{\mathcal{G}}) > 0$ be a Lipschitz constant of $\hat{\mcG}$ around $X$. For every $\beta_1, \beta_2 \in L^2_\mu$ it holds
        \begin{align*}
            |D \hat{\mcG}(X)(\beta_2 \circ X) - D \hat{\mcG}(X)(\beta_1 \circ X)|&\leq \Lip_X(\hat{\mathcal{G}})| \beta_2\circ X - \beta_1 \circ X|_{L^2_\mbP}\\
            &= \Lip_X(\hat{\mathcal{G}}) \norm{ \beta_2 - \beta_1}_\mu.
        \end{align*}
        Hence, it suffices to prove \eqref{eq: equality generalized derivative} for functions $\psi \in C_c^\infty(\R^d;\R^d)$. We compute that
        \begin{align*}
           \limsup_{\substack{Y \to X\\ t \searrow 0 }} \frac{| \hat{\mcG}( Y + t \psi \circ X) -  \hat{\mcG}(Y + t \psi \circ Y) |}{t} &\leq \Lip(\hat{\mcG}) \limsup_{Y \to X} \norm{\psi \circ X - \psi \circ Y}_\mbP\\
           &\leq \Lip(\hat{\mcG}) \Lip(\psi) \limsup_{Y \to X} \norm{Y- X}_\mbP\\
           &= 0.
        \end{align*}
        Consequently,
        \begin{align*}
            D \hat{\mcG}(X)(\psi\circ X) = \limsup_{\substack{Y \to X\\ t \searrow 0 }} \frac{\hat{\mcG}(Y + t\psi \circ Y) - \hat{\mcG}(Y) }{t} &= \limsup_{\substack{Y \to X\\ t \searrow 0 }} \frac{\mcG ( (\id + t\psi)_\# (Y_\# \mbP)) - \mcG(Y_\# \mbP)}{t}\\
            &\leq D\mcG(\mu)(\psi).
        \end{align*}
        To prove the equality, choose $\nu_n \to \mu$ and $t_n \searrow 0$ such that the limsup in the definition of $D\mcG(\mu)(v)$ is a limit. By \cite[Lemma 2.1]{bertucciApproximationSquaredWasserstein2024}, there are random variables $Y_n \to X$ such that $(Y_n)_\# \mbP = \nu_n$. Therefore,
        \begin{align*}
            D\mcG(\mu)(\psi) = \lim_n \frac{\mcG ( (\id + t_n \psi)_\#\nu_n) - \mcG(\nu_n)}{t_n} &\leq \limsup_{\substack{Y \to X\\ t \searrow 0 }}  \frac{\mcG ( (\id + t\psi)_\# (Y_\# \mbP)) - \mcG(Y_\# \mbP)}{t}\\
            &= D \hat{\mcG}(X)(\psi\circ X).
        \end{align*}
        To prove the last part of the statement, since $\partial_C \hat{\mcG}(X)$ is convex, closed and bounded and $\Bar_X[\cdot]$ linear and bounded, the set $\Bar_X[\partial_C \hat{\mcG}(X)]$ is convex and closed. By Proposition \ref{prop:basic-properties-clarke-subdiff} the same holds for $\partial_C \mcG(\mu)$, and in particular it suffices to prove that their support functions agree. Let $\beta \in L^2_\mu$. By the first part of the statement and \eqref{eq: generalized derivative and support function classical case} we have that
        \begin{align*}
            \sigma_{\Bar_X[\partial_C \hat{\mcG}(X)]}(\beta) = \sup_{ \Xi \in \partial_C \hat{\mcG}(X)} \inner{\Bar_X[\Xi], \beta}_\mu = \sup_{ \Xi \in \partial_C \hat{\mcG}(X)} \inner{\Xi, \beta\circ X}_\mbP &= D \hat{\mcG}(X)(\beta \circ X)\\
            &= D \mcG(\mu)(\beta)\\
            &= \sigma_{\partial_C \mcG(\mu)}(\beta).
        \end{align*}
\end{proof}

\begin{corollary}[Clarke subdifferential of a $C^{1,1}$ mapping]\label{coro: clarke subdifferential of C11 mappings}
    Let $\mcG: \sP_2(\R^d) \to \R$ be a locally Lipschitz mapping. If $\mcG$ is of class $C^{1,1}$, it holds that
    $$\partial_C \mcG(\mu) = \{ \nabla_\mu \mcG(\mu)\}.$$
\end{corollary}

\begin{proof}
    If $\mcG$ is of class $C^{1,1}$ on $\sP_2(\R^d)$, its lifting $\hat{\mcG}$ is of class $C^{1,1}$ in $L^2_\mbP$. Fix $X \sim \mu$. Using \cite[Theorem 10.8]{clarke2013functional} and Theorem \ref{thm:gangbo-tudorascu} we have $\partial_C \hat{\mcG}(X) = \{ \nabla_\mu \mcG(\mu) \circ X\}$. Hence, by Proposition \ref{prop: clarke subdifferential and lifting},
    $$\partial_C \mcG(\mu) = \Bar_X[\partial_C \hat{\mcG}(X)] = \{\nabla_\mu \mcG(\mu)\}.$$
\end{proof}

\subsection{Clarke subdifferential and convexity along plans}

The Clarke subdifferential on $\sP_2(\R^d)$ is defined through deterministic perturbations of the form $(\id+t\psi)_\#\mu$. At first sight, this suggests that it can only detect variations of $\mu$ that are induced by maps. We show that, under convexity along plans, this apparent limitation can be overcome. In fact, the Clarke subdifferential of a convex along plans mapping is characterized by a global supporting inequality against arbitrary transport plans; see Theorem \ref{thm: clarke subdiff and convex mappings}.\\

Let us start with the following Young-type inequality. It shows
that, for functionals that are convex along plans, replacing a transport plan by its barycentric projection cannot increase the value of the functional.

\begin{lemma}[Barycentric Fenchel-Young inequality]\label{lemma: barycenters and convexity by plans}
    Let $\mcG: \sP_2(\R^d) \to \R \cup \{+ \infty\}$ be proper, lower semicontinuous and convex along plans. Then, for every $\mu, \nu \in \sP_2(\R^d)$ and $\gamma \in \Gamma(\mu,\nu)$ it holds that
    \begin{align}\label{ineq: barycenters and convexity by plans}
        \mcG(\mcB[\gamma]_\# \mu) \leq \mcG(\pi^2_\# \gamma) =\mcG(\nu).
    \end{align}
\end{lemma}

\begin{proof}
    Suppose first that $\mcG$ is differentiable, and fix $(X,Y) \sim \gamma$. By Theorem \ref{thm:gangbo-tudorascu},
    $$\nabla \hat{\mcG}( \mcB[\gamma] \circ X) = \nabla_\mu \mcG( \mcB[\gamma]_\# \mu) \circ ( \mcB[\gamma] \circ X).$$
    Let $\hat{\mcG}^*$ be the Fenchel conjugate of the convex mapping $\hat{\mcG}$. Fenchel-Young inequality yields
    $$\mcG(\mcB[\gamma]_\# \mu) = \hat{\mcG}( \mcB[\gamma] \circ X) = \inner{\nabla_\mu \mcG( \mcB[\gamma]_\# \mu) \circ ( \mcB[\gamma] \circ X), \mcB[\gamma] \circ X }_\mbP - \hat{\mcG}^*(\nabla \mcG( \mcB[\gamma]_\# \mu) \circ ( \mcB[\gamma] \circ X)).$$
    Again, by Fenchel-Young inequality we deduce that
    \begin{align*}
        \inner{\nabla \mcG( \mcB[\gamma]_\# \mu) \circ ( \mcB[\gamma] \circ X), \mcB[\gamma] \circ X }_\mbP &= \int \inner{\nabla_\mu \mcG( \mcB[\gamma]_\# \mu)(\mcB[\gamma](x)), \mcB[\gamma](x) } d \mu(x)\\
        &= \int \inner{\nabla_\mu \mcG( \mcB[\gamma]_\# \mu)(\mcB[\gamma](x)), y} d \gamma(x,y)\\
        &= \inner{ \nabla_\mu \mcG( \mcB[\gamma]_\# \mu) \circ (\mcB[\gamma] \circ X), Y}_\mbP\\
        &\leq \hat{\mcG}^*\left( \nabla_\mu \mcG( \mcB[\gamma]_\# \mu) \circ (\mcB[\gamma] \circ  X ) \right) + \hat{\mcG}(Y).
    \end{align*}
    Hence,
    $$\mcG( \mcB[\gamma]_\# \mu) \leq \hat{\mcG}(Y) = \mcG(\nu).$$
    For the general case, the same proof of Proposition \ref{prop: LL on Wasserstein} shows that the Moreau-Yosida envelope $\hat{\mcG}_\lambda$ of $\hat{\mcG}$ is law-invariant, and in particular it induces a well-defined mapping $\mcG_\lambda: \sP_2(\R^d) \to \R$ that is differentiable because of Theorem \ref{thm:gangbo-tudorascu}. Moreover, $\mcG_\lambda \to \mcG$ pointwise. Since \eqref{ineq: barycenters and convexity by plans} holds for $\mcG_\lambda$, we can pass to the limit as $\lambda \searrow 0$ and deduce the same inequality for $\mcG$.
\end{proof}

\begin{remark}
    The inequality is nontrivial because $\mcB[\gamma]_\#\mu$ is obtained from $\mu$ through a deterministic map, whereas $\nu=\pi^2_\#\gamma$ may involve mass splitting. Thus, convexity along plans allows one to compare a non-deterministic displacement with its deterministic barycentric counterpart.
\end{remark}

With the help of the previous lemma, we can upgrade the map-based definition of the Clarke
subdifferential into a plan-level statement.

\begin{theorem}[Clarke subdifferential of convex mappings]\label{thm: clarke subdiff and convex mappings}
    Let $\mcG: \sP_2(\R^d) \to \R \cup \{+ \infty\}$ be convex along plans and locally Lipschitz around $\mu \in \sP_2(\R^d)$. Then,
    $$\Xi \in \partial_C \mcG(\mu) \quad \iff \quad \forall  \nu \in \sP_2(\R^d), \, \forall \gamma \in \Gamma(\mu, \nu), \quad \mcG(\mu) + \int_{\R^{2d}} \inner{\Xi(x), y-x} d \gamma(x,y) \leq \mcG(\nu).$$
\end{theorem}

\begin{proof}
    Let $\Xi \in \partial_C \mcG(\mu)$ and $\gamma \in \Gamma(\mu,\nu)$. Choose random variables $X,Y \in L^2_\mbP$ such that $(X,Y)_\#\mbP = \gamma$. Since $\hat{\mcG}$ is convex in $L^2_\mbP$, by \cite[Theorem 10.8]{clarke2013functional} and Proposition \ref{prop: clarke subdifferential and lifting},
    $$\inner{\Xi, \beta - \id}_\mu \leq D \mcG(\mu)(\beta - \id) = D \hat{\mcG}(X)\left( (\beta-\id) \circ X \right) = \lim_{t \searrow 0} \frac{\hat{\mcG}(X + t(\beta - \id) \circ X) - \hat{\mcG}(X)}{t}.$$
    Using the monotonicity of the quotients of the convex mapping $\hat{\mcG}$, we deduce that
    $$\forall \beta \in L^2_\mu, \qquad \inner{\Xi, \beta - \id}_\mu \leq \mcG(\beta_\# \mu) - \mcG(\mu).$$
    Take any $\nu \in \sP_2(\R^d)$ and choose $\gamma \in \Gamma(\mu, \nu)$. Taking $\beta = \mcB[\gamma] \in L^2_\mu$ in the above inequality, using the linearity of the inner product and Lemma \ref{lemma: barycenters and convexity by plans} we deduce that
    $$\int_{\R^{2d}} \inner{\Xi(x), y - x} d \gamma(x,y) = \inner{\Xi, \mcB[\gamma] - \id}_\mu \leq \mcG(\nu) - \mcG(\mu).$$
    The other inclusion is straightforward.
\end{proof}

This global characterization is precisely the property used later in the stability argument for the regularized running cost of Proposition \ref{prop:running-cost-LL-closedness}. As a direct consequence, we have the following.

\begin{corollary}\label{coro: optimality condition convex along plan mappings}
    Let $\mcG: \sP_2(\R^d) \to \R$ be convex along plans and locally Lipschitz around $\mu \in \sP_2(\R^d)$. Then,
    $$\mu \in \operatorname{argmin} \mcG \quad \iff \quad 0 \in \partial_C \mcG(\mu).$$
    \hfill $\blacksquare$
\end{corollary}

Thus, for mappings convex along plans, the Clarke subdifferential has the same global optimality interpretation as the classical subdifferential of a convex function, despite being defined through deterministic perturbations.

\begin{remark}[Sharpness of the convexity along plans]
    The assumption of convexity along plans is essential. The conclusion of Theorem \ref{thm: clarke subdiff and convex mappings} does not, in general, follow from mere $d_\mcW$-geodesic convexity. Indeed, let $G:\sP_2(\R)\to\R$ be defined by
    \[ \mu \in \sP_2(\R) \mapsto
    G(\mu):=d_\mcW^2(\mu,\sigma), \qquad
    \sigma:=\tfrac12\delta_{-1}+\tfrac12\delta_1.
    \]
    The functional $G$ is convex along $d_\mcW$-geodesics (see \cite[Theorem 10.18]{ambrosioLecturesOptimalTransport2021}), and Theorem \ref{thm: clarke subdiff of wass dist} yields
    \[
    \partial_C G(\delta_0)=\{0\}.
    \]
    Still, $\delta_0$ is not a minimizer of $G$. To interpret this phenomenon, we recall that every perturbation of $\delta_0$
    detected by $\partial_C$ is of the form
    \[
    (\id+t\psi)_\#\delta_0=\delta_{t\psi(0)},
    \]
    and hence keeps the measure concentrated at a single point, while the $d_\mcW$-geodesic from $\delta_0$ to $\sigma$ is given by
    \[
    \mu_t=\frac12\delta_{-t}+\frac12\delta_t,
    \qquad t\in[0,1],
    \]
    which immediately splits the atom at the origin. Such a variation cannot
    be represented by a deterministic perturbation of $\delta_0$ and is
    therefore invisible to $\partial_C$. This shows both the limitation of map-based variations and the genuine role of convexity along plans in Theorem \ref{thm: clarke subdiff and convex mappings}.
\end{remark}

\section{Lasry-Lions regularization on $\sP_2(\R^d)$}\label{sec: Lasry Lions on Wasserstein}
The Lasry-Lions regularization, introduced in \cite{lasry1986remark}, is a useful tool to regularize functions defined on Hilbert spaces. In this section, with the help of the lifting technique, we show how this regularization extends naturally to $\sP_2(\R^d)$ and its main properties, including relations with the Clarke subdifferential of Definition \ref{defi: clarke subdifferential}.

\subsection{Definition and general properties}
Let $U: L^2_\mbP \to \R$ be a Lipschitz function. The Lasry-Lions regularization $U^\epsilon: L^2_\mbP \to \R$ of $U$ is defined as
$$U^\epsilon(X):= \inf_{Z \in L^2_\mbP} \sup_{Y \in L^2_\mbP} \bigg\{ U(Y) - \frac{1}{2 \epsilon} |Z-Y|^2 + \frac{1}{\epsilon} |Z-X|^2 \bigg\}.$$

\begin{definition}\label{def:lasry-lions on Wasserstein}
    Let $\mcG: \sP_2(\R^d) \to \R$ be Lipschitz. The Lasry-Lions regularization $\mcG^\epsilon: \sP_2(\R^d)$ of $\mcG$ is defined as the map
    $$\mu \in \sP_2(\R^d) \mapsto \mcG^\epsilon(\mu) := \hat{\mcG}^\epsilon(X) \quad \text{where} \quad X_\# \mbP = \mu.$$
\end{definition}

In the proposition below we prove that $\mcG^\epsilon$ is well defined, in the sense that it does not depend on the chosen random variable $X \sim \mu$.

\begin{proposition}\label{prop: LL on Wasserstein}
    Let $\epsilon >0$, $\mcG: \sP_2(\R^d) \to \R$ be a Lipschitz function and let $\hat{\mcG}^\epsilon: L^2_\mbP \to \R$ be the Lasry-Lions regularization of the lifted function $\hat{\mcG}: L^2_\mbP \to \R$. Then, $\hat{\mcG}^\epsilon(X)$ only depends on $X_\# \mbP$. In particular, the mapping $\mcG^\epsilon: \sP_2(\R^d) \to \R$ given by
    $$\mcG^\epsilon(\mu) := \hat{\mcG}^\epsilon(X), \quad \text{where} \quad X_\# \mbP = \mu$$
    is well defined. Moreover, $\mcG^\epsilon$ is of class $C^{1,1}$ and
    \begin{align}
        \Lip(\mcG^\epsilon) \leq \Lip(\mcG), \qquad \qquad 0 \leq \mcG^\epsilon - \mcG \leq \frac{\Lip(\mcG)^2}{2} \epsilon.
    \end{align}
    Furthermore, if $\mcG$ is convex along plans, so is $\mcG^\epsilon$.
\end{proposition}

\begin{proof}
    Let \(\tau:\Omega\to\Omega\) be a bimeasurable measure-preserving bijection on $(\Omega; \mbP)$. Since \(\hat{\mcG}\) is law-invariant and composition with \(\tau\) is an isometry of \(L^2_\mbP\), a change of variables in the variational formula defining \(\hat{\mcG}^\varepsilon\) gives
    \begin{align}\label{eq: LL regularization and measure preserving maps}
            \hat{\mcG}^\varepsilon(X\circ\tau)=\hat{\mcG}^\varepsilon(X) \qquad \forall X\in L^2_\mbP.
    \end{align}
    Let \(X^1, X^2\in L^2_\mbP\) with the same law. By \cite[Lemma 5.23]{carmonaProbabilisticTheoryMean2018}, there exist bimeasurable measure-preserving (almost) bijections \((\tau_n)_n\) such that $X^1 \circ \tau_n \to X^2$ strongly in $L^2_\mbP$. Passing to the limit in \eqref{eq: LL regularization and measure preserving maps} and using the continuity of $\hat{\mcG}^\epsilon$ we deduce that
    $$\hat{\mcG}^\epsilon(X^2) = \hat{\mcG}^\epsilon(X^1).$$
    The remaining part of the statement follows \cite{lasry1986remark}.
\end{proof}

\begin{blue}
    
\end{blue}
\begin{remark}
    One of the main features of the Lasry-Lions is that it not only regularizes $\mcG$, but also satisfies $\mcG^\epsilon \geq \mcG$ and converges uniformly on $\sP_2(\R^d)$ (and not only locally uniformly). These properties will be used in the proof of Theorem \ref{thm: non smooth pmp}.
\end{remark}

\subsection{Relations with Clarke subdifferential}
\hfill

It is known that the Lasry-Lions approximation also provides first-order information of the regularized functional. More specifically, in \cite{benoist1992convergence} it is proven that weak limits of gradients of $U^\epsilon$ belong to the Clarke subdifferential of $U$ at a given point. This property can be extended to the Wasserstein space $\sP_2(\R^d)$.

\begin{proposition}[Clarke subdifferential and Lasry-Lions regularization]
    \label{prop: Clarke subdiff and Lasry Lions}
    Let \( \mcG:\P_2(\R^d)\to\R\) be a Lipschitz continuous function and let \(\mcG^\epsilon\) be its Lasry-Lions regularization. Let $\mu^\epsilon \to \mu$ in $\sP_2(\R^d)$ and set
    $$\Xi^\varepsilon := \nabla_\mu \mcG^\epsilon(\mu^\epsilon) \in L^2_{\mu^\epsilon}(\R^d;\R^d).$$
    Let $\gamma^\varepsilon\in\Gamma(\mu,\mu^\epsilon)$ be such that $\norm{\pi^2 - \pi^1}_{\gamma^\epsilon} \to 0$ as $\epsilon \searrow 0$. Then, $\{\PP_{\gamma^\varepsilon}\Xi^\varepsilon\}_\epsilon$ is bounded in $L^2_\mu$ and any of its weak limit points belongs to $\partial_C \mcG(\mu)$.
\end{proposition}

\begin{proof}
    We first see that Propositions \ref{prop: LL on Wasserstein} and \ref{prop:basic-properties-clarke-subdiff} together with \eqref{ineq: bounds-pullback} yield,
    \begin{align}\label{eq: boundedness gradients Lasry Lions}
        \|\mathsf P_{\gamma^\varepsilon}\Xi^\varepsilon\|_{L^2_\mu} \leq 
        \|\Xi^\varepsilon\|_{L^2_{\mu^\epsilon}} \leq \Lip(\mcG).
    \end{align}
    In particular, the set $\{\PP_{\gamma^\varepsilon}\Xi^\varepsilon\}_\epsilon$ admits cluster points in the weak topology of $L^2_\mu$. Let $\xi$ be any of these points and, up to relabeling, write $\mathsf P_{\gamma^\varepsilon}\Xi^\varepsilon
    \rightharpoonup \xi$.
    
    Fix $X \in L^2_\mbP$ with \(X_\# \mbP=\mu\). By \cite[Lemma 2.1]{bertucciApproximationSquaredWasserstein2024}, there are random variables $X^\epsilon \to X$ such that $X^\epsilon_\# \mbP = \mu^\epsilon$. By Theorem \ref{thm:gangbo-tudorascu}, $\nabla \hat{\mcG}^\epsilon(X^\epsilon) = \Xi^\epsilon \circ X^\epsilon$. Moreover,  by \eqref{eq: boundedness gradients Lasry Lions} we can suppose without loss of generality that there is $\Xi \in L^2_\mbP$ such that $\Xi^\epsilon \circ X^\epsilon \rightharpoonup \Xi$ weakly in $L^2_\mbP$. By \cite[Theorem 3]{benoist1992convergence}\footnote{In the cited article, the authors prove the result for the $\sup-\inf$ Lasry-Lions regularization. However, using the property $\partial_C(-U) = - \partial_C U$ together with the relations between the $\sup-\inf$ regularization and the $\inf-\sup$ regularization, the result holds for the $\inf-\sup$ regularization as well.} , we deduce that $\Xi \in \partial_C\hat{\mcG}(X)$.
    If we prove that $\xi =\Bar_X[\Xi]$, Proposition \ref{prop: clarke subdifferential and lifting} would yield that $\xi \in \partial_C \mcG(\mu)$. To do so, it suffices to prove that
    $$\forall \psi \in C_c^\infty(\R^d;\R^d), \qquad \inner{\xi, \psi}_\mu = \inner{\Xi, \psi \circ X}_\mbP.$$
    
    By definition of the pullback mapping,
    \begin{align*}
        \inner{\xi, \psi}_\mu = \lim_{\epsilon \searrow 0} \inner{\mathsf P_{\gamma^\varepsilon}\Xi^\varepsilon, \psi }_\mu &= \lim_{\epsilon \searrow 0} \int \inner{\Xi^\epsilon(y), \psi(x)} d \gamma^\epsilon(x,y).
    \end{align*}
    
    We see that
    \begin{align*}
        \bigg| \int \inner{\Xi^\epsilon(y), \psi(x)} - \inner{\Xi^\epsilon(y), \psi(y)}   d \gamma^\epsilon(x,y) \bigg| \leq \norm{\Xi^\epsilon}_{\mu^\epsilon} \Lip(\psi) \norm{\pi^2 - \pi^1}_{L^2_{\gamma^\epsilon}} \xrightarrow[\epsilon \searrow 0]{} 0.
    \end{align*}
    
    Since $X^\epsilon \to X$ strongly, repeating the same argument as above and using the weak convergence $\Xi^\epsilon \circ X^\epsilon \rightharpoonup \Xi$, it follows that
    \begin{align*}
        \inner{\xi, \psi}_\mu = \lim_{\epsilon \searrow 0} \inner{\Xi^\epsilon, \psi}_{\mu^\epsilon} = \lim_{\epsilon \searrow 0} \inner{\Xi^\epsilon\circ X^\epsilon, \psi \circ X^\epsilon}_\mbP &= \lim_{\epsilon \searrow 0} \inner{\Xi^\epsilon \circ X^\epsilon, \psi \circ X}_\mbP = \inner{\Xi, \psi \circ X}_\mbP.
    \end{align*}
\end{proof}

\begin{proposition}\label{prop:running-cost-LL-closedness}
    Suppose that $L$ satisfies [A\ref{hyp: assumptions lagrangian}.i,ii] and that it is convex along plans. Let $(t^\epsilon, u^\epsilon, \mu^\epsilon) \to (t, u, \mu)$ and consider $\gamma^\epsilon \in \Gamma(\mu, \mu^\epsilon)$ such that $\norm{\pi^2 - \pi^1}_{\gamma^\epsilon} \to 0$. Let $L^\epsilon$ be the Lasry-Lions regularization of $L$ w.r.t. $\mu$, and set $\Xi^\epsilon := \nabla_\mu L^\epsilon(t^\epsilon, u^\epsilon, \mu^\epsilon)$. Then, any limit point $\Xi$ of $\PP_{\gamma^\epsilon} \Xi^\epsilon$ in the weak topology of $L^2_\mu$ (which always exists) is such that
    $$\Xi \in \partial_C L(t,u,\cdot)(\mu).$$
\end{proposition}

\begin{proof}
    Let $\nu \in \sP_2(\R^d)$ and $\eta \in \Gamma(\mu, \nu)$. Disintegrate $\eta = \mu(dx) \eta_x(dz)$ and $\gamma^\epsilon = \mu(dx) \gamma^\epsilon_x(dy)$ and define the plan $\eta^\epsilon := \mu(dx) \gamma^\epsilon_x \otimes \eta_x(dy,dz)$. Since $L^\epsilon(t,u,\cdot)$ is convex along plans, by Proposition \ref{prop: LL on Wasserstein} and Theorem \ref{thm: clarke subdiff and convex mappings} we deduce that
    \begin{align}\label{eq: L epsilon inequality}
        L^\epsilon(t^\epsilon, u^\epsilon, \mu^\epsilon) + \int_{\R^{3d}} \inner{\Xi^\epsilon(y), z-y} d \eta^\epsilon(x,y,z) \leq L^\epsilon(t^\epsilon, u^\epsilon, \nu)
    \end{align}
    Disintegrating and applying the definition of $\eta^\epsilon$,
    $$\int_{\R^{2d}} \inner{\Xi^\epsilon(y), z-y} d \eta^\epsilon(x,y,z) = \int_{\R^{2d}} \inner{\Xi^\epsilon(y), \mcB[\eta](x) - y} d \gamma^\epsilon(x,y).$$
    Hölder's inequality, Proposition \ref{prop:basic-properties-clarke-subdiff} and the definition of pullback mapping yield
    \begin{align*}
        \int_{\R^{2d}} \inner{\Xi^\epsilon(y), \mcB[\eta](x) - y} d \gamma^\epsilon(x,y) &= \int_{\R^{2d}} \inner{\Xi^\epsilon(y), \mcB[\eta](x) - x} d \gamma^\epsilon(x,y) + \norm{\Xi^\epsilon}_{L^2_{\mu^\epsilon}} \norm{\pi^2 - \pi^1}_{\gamma^\epsilon}\\
        &= \int_{\R^d} \inner{\PP_{\gamma^\epsilon} \Xi^\epsilon(x), \mcB[\eta](x) - x} d \mu(x) + \norm{\Xi^\epsilon}_{L^2_{\mu^\epsilon}} \norm{\pi^2 - \pi^1}_{\gamma^\epsilon}.
    \end{align*}
    Using the weak convergence $\PP_{\gamma^\epsilon} \Xi^\epsilon \rightharpoonup \Xi$ together with Proposition \ref{prop:basic-properties-clarke-subdiff}.vi) and Proposition \ref{prop: LL on Wasserstein}, we deduce that
    $$\lim_{\epsilon \searrow 0} \int_{\R^{2d}} \inner{\Xi^\epsilon(y), z-y} d \eta^\epsilon(x,y,z) = \int_{\R^d} \inner{\Xi(x), \mcB[\eta](x) - x} d \mu(x) = \int_{\R^{2d}} \inner{\Xi(x), y-x} d \eta(x,y).$$
    Proposition \ref{prop: LL on Wasserstein}, \eqref{eq: L epsilon inequality} and the continuity of [A\ref{hyp: assumptions lagrangian}.i] provide
    $$L(t^\epsilon, u^\epsilon, \mu^\epsilon) +  \int_{\R^{3d}} \inner{\Xi^\epsilon(y), z-y} d \eta^\epsilon(x,y,z) \leq \frac{C}{2} \epsilon +  L(t^\epsilon, u^\epsilon, \nu).$$
    Letting $\epsilon \searrow 0$ we deduce that
    $$\forall \nu \in \sP_2(\R^d), \forall \eta \in \Gamma(\mu, \nu), \qquad L(t,u, \mu) + \int_{\R^{2d}} \inner{\Xi(x), y-x} d \eta(x,y) \leq L(t,u,\nu).$$
    For $\psi \in C^\infty_c(\R^d;\R^d)$, letting $\nu = (\id + s \psi)_\# \mu$ and $\eta = (\id, \id + s \psi)_\# \mu$,
    $$\inner{\Xi, \psi}_\mu \leq \limsup_{s \searrow 0} \frac{L(t,u, (\id+ s \psi)_\# \mu) - L(t,u,\mu)}{s} \leq D L(t,u,\cdot)(\mu)(\psi).$$
    Therefore, $\Xi \in \partial_C L(t,u,\cdot)(\mu)$.
\end{proof}

\section{Pontryagin maximum principle}\label{sec: PMP on Wasserstein}

In this section we prove Theorem \ref{thm: non smooth pmp}. During this section, for each admissible control $u\in\mathcal U_{ad}$ we denote by $\mu^u$ the associated trajectory, while $g_u$ stands for the mapping
\begin{align}\label{eq: definition g sub u}
    g_u(\cdot) := \bigg( t \mapsto g(\mu^u_t) \bigg) \in C([0,T])
\end{align}

Let $(\mu^*,u^*)$ be a local minimizer of radius $\tilde{\delta}>0$ for~\eqref{prob:P}. The stability ~\eqref{eq:stability-continuity-equation} shows that there is a $\delta > 0$ such that $(\mu^*, u^*)$ is a global minimum over the optimal control problem
\begin{align*}
    \min_{u \in\mcU_{ad} \cap B_{L^2}[u^*,\delta]} & \mcJ(\mu^u(T)) + \int_0^T L(t, u(t), \mu^u(t))\,\mathrm{d}t \\
    \text{s.t} \quad &(\mu^u,u) \ \text{solves }   \left\{
      \begin{array}{l}
      \partial_t\mu_t+\operatorname{div}\!\big(f[\mu_t](t,u_t,\cdot)\,\mu_t\big)=0\\
      \mu_{|t=0}=\mu_0,
      \end{array}
      \right., \\
    &g(\mu^u(t)) \leq 0 \quad \forall t \in [0,T].
\end{align*}
Setting $\mcU_\delta
:=
\mcU_{ad}\cap B_{d_{L^2}}[u^*,\delta]$, we have that $(\mu^*, u^*)$ is a global minimum of the optimal control problem where the set of controls is $\mcU_{\delta}$.\\

The stability \eqref{eq:stability-continuity-equation} yields that we can shrink $\delta$ and suppose that $\mcJ$ and $g$ are both Lipschitz continuous on the set
$$\mathcal{A}_\delta :=\{\mu^u(t) : t \in [0,T], u \in \mcU_{\delta},  (\mu^u, u) \text{ solves } \eqref{eq: continuity eq}\}.$$
Hence, the McShane-Whitney theorem (cf. \cite{mcshane1934extension}) provides that $\mcJ$ and $g$ admit both Lipschitz extensions to $\sP_2(\R^d)$ that agree with $\mcJ$ and $g$ on open set $\mcO$ containing $\mathcal A_\delta$. Similarly, shrinking again $\delta$ if necessary, by [A\ref{hyp: assumptions lagrangian}.ii], we can take $C > 0$ such that for every $(t,u)$, the mapping $L(t,u,\cdot)$ is $C$-Lipschitz on an open set $\mcO$ containing $\mcA$.  Using the explicit expression of the McShane extension of \cite{mcshane1934extension} we can re-define $L$ as
$$L(t,u,\mu) := \sup_{\sigma \in \mcO} \left\{ L(t,u, \sigma) -  C d_\mcW(\mu, \sigma)\right\}.$$
This mapping agrees with the original function on $[0,T] \times U \times \mcO$, is $C$-Lipschitz on $\sP_2(\R^d)$ for every $(t,u)$, and the separability of $\sP_2(\R^d)$ shows that $L(\cdot, \mu)$ is Borel measurable. Moreover, since the Clarke subdifferential is a local object, the inclusions described by Theorem \ref{thm: non smooth pmp} still hold for the original functions $\mcJ$, $g$, $L$ if they hold for their extensions.\\

Consequently, in what follows we suppose without loss of generality $\mcJ$ and $g$ are globally Lipschitz and that [A\ref{hyp: assumptions lagrangian}.ii] holds globally.

\begin{remark}
    We point out that Proposition \ref{prop:running-cost-LL-closedness} still holds for the extension of $L$ since the arguments presented there are local. In fact, starting the proof we can take $\nu \in \mcO$ and replace the extension by the original function $L$ that is convex along plans.
\end{remark}

\subsection{Diffusive perturbation}
As already mentioned in the introduction, the compact-support assumption on $\supp(\mu_0)$ plays a structural role in the existing approaches, rather than being merely technical. In \cite{bonnet2019pontryagin, bonnet2021necessary, bonnetrossi2019}, it is needed to uniformize the error term in the Taylor expansion of the perturbed trajectories, while the technique employed in \cite{urrea2026pontryagin} relies on the existence of a finite covering of $\supp(\mu_0)$ by arbitrarily small balls (see \cite[Proposition 5.1]{urrea2026pontryagin}) to guarantee that the diffusive perturbation is of order $o(\rho)$. We overcome these difficulties by refining the diffusive perturbation argument of \cite{urrea2026pontryagin} so as to avoid any compactness assumption on $\supp(\mu_0)$. The argument does not require the control space to be a subset of a Euclidean space, and therefore extends to the separable metric space $(U,d_U)$ considered here and to any initial condition \(\mu_0\in\sP_2(\R^d)\). We state the resulting perturbation estimate below and defer its proof to Appendix~\ref{sec: diffusive perturbation}.

\begin{proposition}[Diffuse linearization]
    \label{prop:diffuse-linearization}
    
    Assume [A\ref{hyp:dynamic-non-smooth}] and [A\ref{hyp: assumptions lagrangian}.i,ii]. Let \(\bar u,u\in\mathcal U_{\rm ad}\) and denote by \(\bar\mu\) the trajectory associated with \(\bar u\) and by
    \(\Phi^{\bar u}_{(0,t)}\) its flow. Then, for every $\epsilon > 0$ and any \(\rho \in (0,1)\), there exists a measurable set
    \(E_\rho\subset[0,T]\), with $\sL(E_\rho) = \rho T$ such that for the control
    $$u^\rho := \begin{cases}
        \bar{u} \quad \text{ in } [0,T] \setminus E_\rho\\
        u \quad \text{ in } E_\rho
    \end{cases}$$
    and its corresponding flow \(\Phi^\rho_{(0,t)}\) it holds that
    \[
    \Phi^\rho_{(0,\cdot)}
    =
    \Phi^{\bar u}_{(0,\cdot)}
    +\rho\Psi
    +o(\rho)
    \qquad
    \text{in }C([0,T]; L^2_{\mu_0}),
    \]
    and
    \begin{equation}\label{eq:Lyap-L}
    \int_{E_\rho} \bigg( L^\epsilon(t, u^\rho_t, \bar{\mu}_t)  - L^\epsilon(t, \bar{u}_t, \bar{\mu}_t) \bigg) \, dt = \rho \int_0^T \bigg( L^\epsilon(t, u_t, \bar{\mu}_t) - L^\epsilon(t, \bar{u}_t, \bar{\mu}_t) \bigg) \, dt
    \end{equation}
    where \(\Psi\in AC([0,T]; L^2_{\mu_0})\) is the unique solution, for
    \(\mu_0\)-a.e. \(x\), of
    \begin{align}\label{eq: linearized equation}
        \begin{cases}
            \partial_t \Psi(t,x) &=\quad f[\bar{\mu}_t](t, u_t, \Phi^{\bar{u}}_{(0,t)}(x)) - f[\bar{\mu}_t](t,\bar{u}_t, \Phi^{\bar{u}}_{(0,t)}(x))\\
            &+\quad D_x f[\bar{\mu}_t](t, \bar{u}_t, \Phi^{\bar{u}}_{(0,t)}(x)) \cdot \Psi(t, x)\\
            &+\quad \displaystyle\int_{\R^d} \bigg( \nabla_\mu f[\bar{\mu}_t](t, \bar{u}_t, \Phi^{\bar{u}}_{(0,t)}(x))(\Phi^{\bar{u}}_{(0,t)}(y) ) \cdot \Psi(t,y) \bigg) \, d \mu_0(y)\\
            \Psi(0,x) &=\quad 0
        \end{cases}.
    \end{align}
\end{proposition}

\begin{remark}
    The above claim also holds for $L$ instead of $L^\epsilon$, and the only requirement for $L$ is that it must be $\sL$-measurable on $t$ and continuous in $(u, \mu)$.
\end{remark}

\begin{proposition}\label{prop: taylor expansion costs}
    Let $\bar{\mu}$ be the unique solution of \eqref{eq: continuity eq} driven by a control $\bar{u} \in \mcU_{ad}$ and let $u^\rho \in \mcU_{ad}$ and $\mu^\rho$ as in Proposition \ref{prop:diffuse-linearization}. Under [A\ref{hyp:dynamic-non-smooth}]-[A\ref{hyp: assumptions lagrangian}], for every $\epsilon, \rho > 0$ it holds that
    \begin{align*}
        \mcJ^\epsilon(\mu^\rho_T) = \mcJ^\epsilon(\bar{\mu}_T) + \rho \int_{\R^d} \left\langle \nabla_\mu \mcJ^\epsilon(\bar{\mu}_T) \left(\Phi^{\bar{u}}_{(0,T)}(x) \right), \Psi(T,x) \right\rangle d \mu_0(x) + o(\rho)
    \end{align*}
    and
    \begin{align*}        
        \int_0^T L^\epsilon(t, u^\rho_t, \mu^\rho_t) \, dt &= \int_0^T L^\epsilon(t, \bar{u}_t, \bar{\mu}_t) \, dt +\rho \bigg[ \int_0^T \bigg( L^\epsilon(t, u_t, \bar{\mu}_t) - L^\epsilon (t, \bar{u}_t, \bar{\mu}_t) \bigg) \, dt  \bigg]\\
        &+ \rho \int_0^T \int_{\R^d} \bigg( \left \langle \nabla_\mu L^\epsilon (t, \bar{u}_t, \bar{\mu}_t)\left( \Phi^{\bar{u}}_{(0,t)}(x) \right), \Psi(t,x) \right \rangle \, d\mu_0(x) \bigg) dt + o (\rho)
    \end{align*}
    where $\Psi$ solves \eqref{eq: linearized equation}.
\end{proposition}

\subsection{Proof of Theorem~\ref{thm: non smooth pmp}}

 Let $d_C$ denote the distance to the set $C := \{z\in C([0,T]): z(t)\le 0 \ \forall t\in[0,T]\}$ with the norm

$|\cdot|_{C([0,T])}$, i.e.,
\[
d_C(h):=\inf_{z\in C}|h-z|_{C([0,T])}.
\]
We shall use a smooth equivalent norm on \(C([0,T])\). More precisely, since
\(C([0,T])\) is separable, we fix an equivalent norm
\(|\cdot|_{C([0,T])}\) such that \(C([0,T])\) and its dual
\(\mathcal M([0,T])\), endowed with the corresponding dual norm, are strictly
convex. 
Then \(d_C\) is convex and \(1\)-Lipschitz. Moreover, \(d_C\) is
Gâteaux differentiable on \(C([0,T])\setminus C\). Equivalently, if
\(h\notin C\), then \(\partial d_C(h)\) is a singleton, and we denote its
unique element by
\[
\nabla d_C(h)\in\mathcal M([0,T]).
\]
In this case,
\(
|\nabla d_C(h)|_{\mathcal M([0,T])}=1
\).

\medskip

{\bf Step 1 }(Approximation).

For every $\epsilon > 0$ set
\[
\mcU_{ad}^\epsilon
:=
\left\{
u\in\mathcal U_\delta:\ d_U(u(t), u^*(t))\le \frac{1}{\epsilon^{1/8}}
\ \text{for a.e. }t\in[0,T]
\right\}.
\]
We endow \(\mcU_{ad}^\epsilon\) with the Ekeland distance
\[
d_E(u_1,u_2)
:=
\mathscr L\big(\{t\in[0,T]:u_1(t)\neq u_2(t)\}\big).
\]
By definition,
\begin{align}\label{ineq: ekeland and L2 distance}
    d_{L^2}(u_2, u_1) \leq \frac{2}{\epsilon^{1/8}} d_E^{1/2}(u_2, u_1), \qquad u_1, u_2 \in \mcU^\epsilon_{ad}.
\end{align}
Consequently, \((\mcU^\epsilon_{ad},d_E)\) a complete metric space (see \cite[Lemma 7.2]{ekeland1974variational}, $U$ is not required to be compact).

\smallskip

Let $\mcJ^\varepsilon$, $g^\varepsilon$, and
$L^\varepsilon(t,u,\cdot)$ be the Lasry-Lions regularizations on
$\mathscr P_2(\R^d)$ of $\mcJ$, $g$, and $L(t,u,\cdot)$,
respectively, and set
\[
J^\varepsilon(u)
:=
\mcJ^\varepsilon(\mu_T^u)
+
\int_0^T
L^\varepsilon(t,u_t,\mu_t^u)\,dt.
\]

Define $\mathscr{J}_\epsilon: \mcU^\epsilon_{ad} \to \R$ as
\[ \mathscr{J}_\epsilon(u) :=
\left(
\left([J^\epsilon(u)-J(u^*)+\epsilon]^+\right)^2
+
d_C(g^\epsilon \circ \mu^u)^2
\right)^{1/2}.
\]
The stability \eqref{eq:stability-continuity-equation}, Proposition \ref{prop: LL on Wasserstein} and \eqref{ineq: ekeland and L2 distance} yield that $\mathscr{J}_\epsilon$ is continuous on $\mcU^\epsilon_{ad}$.

\hfill

{\bf Step 2} (Ekeland principle).

\hfill

Under Assumption \ref{hyp:cost-non-smooth}, a direct application of the Lipschitz continuity of $d_C$ together with Proposition~\ref{prop: LL on Wasserstein} yields that there is $C > 0$, depending only on $\mcJ,g,L$ such that
\begin{align}\label{ineq:epsilon minimizer}
    0 \leq \mathscr{J}_\epsilon(u^*) \leq C \epsilon.
\end{align}

In particular, $u^*$ is a $C\epsilon$ minimizer of $\mathscr{J}_\epsilon$ over $\mcU^\epsilon_{ad}$. Applying Ekeland's variational principle we deduce the existence of $u^\epsilon \in \mcU^\epsilon_{ad}$ with $d_E(u^{\epsilon},u^*)\le \sqrt{C \epsilon}$ such that
\begin{align}\label{eq: ekeland variational ineq}
    u^\epsilon \in \operatorname{argmin} \bigg[ u \in \mcU^{\epsilon}_{ad} \mapsto \mathscr{J}_\epsilon(u) + \sqrt{C \epsilon} d_E(u, u^\epsilon) \bigg]
\end{align}

Moreover, \eqref{ineq: ekeland and L2 distance} yields that
$$d_{L^2}(u^\epsilon, u^*)\leq \frac{2}{\epsilon^{1/8}} d_E^{1/2}(u^\epsilon, u^*) \leq 2 C^{1/4} \epsilon^{1/4}.$$
In particular, $u^\epsilon \to u^*$ in $L^2((0,T);U)$.

\hfill

{\bf Step 3} (Approximated optimality conditions).

\hfill

In this step, we prove that for every $v \in \mcU_{ad},$ it holds that
\begin{equation}\label{eq: taylor expansion epsilon cost}
    \begin{aligned}
        0 &\leq \alpha^\epsilon \bigg[ \int_{\R^d}\Big\langle
        \nabla_\mu\mcJ^\epsilon(\mu^\epsilon_T)
        \big(\Phi^\epsilon_{(0,T)}(x)\big),
        \Psi_v^\epsilon(T,x)
        \Big\rangle\, d\mu_0(x)\\
        &+
        \int_0^T \Big( L^{\epsilon}(t,v_\epsilon(t),\mu^\epsilon_t) - L^{\epsilon}(t,u^\epsilon_t,\mu^\epsilon_t) \Big)\,dt\\
        &+
        \int_0^T\int_{\R^d} \Big\langle \nabla_\mu L^{\epsilon}(t,u^\epsilon_t,\mu^\epsilon_t) \big(\Phi^\epsilon_{(0,t)}(x)\big), \Psi_v^\epsilon(t,x)\Big\rangle\,d\mu_0(x)\,dt \bigg]\\
        &+ \int_{[0,T]} \int_{\R^d}
        \Big\langle
        \nabla_\mu g^\epsilon(\mu^\epsilon_t)
        \big(\Phi^\epsilon_{(0,t)}(x)\big),
        \Psi_v^\epsilon(t,x)
        \Big\rangle\,d\mu_0(x) \lambda^\epsilon(dt)\\
        &+ T\sqrt{C \epsilon},
        \end{aligned}
\end{equation}

where
\[\alpha^\epsilon := \frac{1}{\mathscr{J}_\epsilon(u^\epsilon)} [J^{\epsilon}(u^\epsilon)-J(u^*)+\epsilon]^+,
\qquad
\lambda^\epsilon
:=
\begin{cases}
\dfrac{d_C(g^\epsilon \circ\mu^\epsilon)}
{\mathscr{J}_\epsilon(u^\epsilon)}
\,\nabla d_C(g^\epsilon \circ\mu^\epsilon),
&\text{if } d_C(g^\epsilon \circ\mu^\epsilon)>0,\\[1.1em]
0,
&\text{if } d_C(g^\epsilon \circ \mu^\epsilon)=0.
\end{cases}
  \]
with
\[ \alpha^\epsilon\ge0,\qquad \lambda^\epsilon\in \mcM_+([0,T]), \qquad  |\alpha^\epsilon|^2+ |\lambda^\epsilon|_{\mcM([0,T])}^2=1,
\]
and where $\Phi^\epsilon_{(0,\cdot)}$ is the flow associated to the trajectory $(\mu^\epsilon, u^\epsilon)$ and $\Psi_v^\epsilon$ is the unique solution (c.f. Proposition \ref{prop:well-posedness-adjoint}) of the equation
\begin{align}\label{eq: linearized equation truncated}
    \begin{cases}
        \partial_t\Psi(t,x)
          &= f[\mu^\epsilon_t]\bigl(t,v_\epsilon(t),\Phi^\epsilon_{(0,t)}(x)\bigr) - f[\mu^\epsilon_t]\bigl(t,u^\epsilon_t,\Phi^\epsilon_{(0,t)}(x)\bigr)\\
          &\quad+ D_xf[\mu^\epsilon_t] \bigl(t,u^\epsilon_t,\Phi^\epsilon_{(0,t)}(x)\bigr)\Psi(t,x)\\
          &\quad+ \int_{\R^d} \Big\langle \nabla_\mu f[\mu^\epsilon_t] \bigl(t,u^\epsilon_t,\Phi^\epsilon_{(0,t)}(x)\bigr) \bigl(\Phi^\epsilon_{(0,t)}(y)\bigr), \Psi(t,y)\Big\rangle\,d\mu_0(y)\\
          \Psi(0,x)&=0
    \end{cases}
\end{align}
with
\begin{align}\label{eq: truncation of controls}
    v_\epsilon(t):= \begin{cases}
    v_t &\text{ if }\quad  d_U(v_t, u^*_t) \le \frac{1}{\epsilon^{1/8}}\\
    u^*_t &\text{ if } \quad d_U(v_t, u^*_t)> \frac{1}{\epsilon^{1/8}}
\end{cases}.
\end{align}

To prove it, we argue as follows.

\vspace*{0.5cm}

Fix \(v\in\mcU_{ad}\), then \(v_\varepsilon(t)\in\mcU^\varepsilon_{ad}\). Let $E^\epsilon_\rho \subset [0,T]$ be the set given by Proposition~\ref{prop:diffuse-linearization} with reference control
\(u_t^\epsilon\) and comparison control \(v_\epsilon(t)\), and set

$$u^\epsilon_\rho(t) := \begin{cases}
        u^\epsilon_t \quad \text{ in } [0,T] \setminus E^\epsilon_\rho\\
        v_\epsilon(t) \quad \text{ in } E^\epsilon_\rho
    \end{cases}$$
For $\rho > 0$ small enough, \(u_\rho^\epsilon\in \mcU^\epsilon_{ad}\). Hence, \eqref{eq: ekeland variational ineq} yields that
\begin{align}\label{eq: almost taylor expansion ekeland principle}
    0 \leq  \frac{1}{\rho} \bigg(\mathscr{J}_\epsilon(u_\rho^\epsilon) - \mathscr{J}_\epsilon(u^\epsilon) + \sqrt{C \epsilon}\,
d_E(u_\rho^\epsilon,u^\epsilon) \bigg) \leq \frac{\mathscr{J}_\epsilon(u_\rho^\epsilon) - \mathscr{J}_\epsilon(u^\epsilon)}{\rho} +  T\sqrt{C \epsilon}.
\end{align}

Let
\(\mu_\rho^\epsilon\) be the trajectory associated with
\(u_\rho^\epsilon\). Applying Proposition \ref{prop: taylor expansion costs} we deduce that
\begin{equation}\label{eq:first-order-J-vareps}
    \begin{aligned}
          J^{\epsilon}(u_\rho^\epsilon) - J^{\epsilon}(u^\epsilon) &= \rho \bigg[\int_0^T \Big( L^{\epsilon}(t,v_\epsilon(t),\mu^\epsilon_t) - L^{\epsilon}(t,u^\epsilon_t,\mu^\epsilon_t) \Big)\,dt +  \int_{\R^d}\Big\langle
    \nabla_\mu\mcJ^\epsilon(\mu^\epsilon_T)
    \big(\Phi^\epsilon_{(0,T)}(x)\big),
    \Psi_v^\epsilon(T,x)
    \Big\rangle\, d\mu_0(x) \bigg]\\
    &+ \rho\int_0^T\int_{\R^d} \Big\langle \nabla_\mu L^{\epsilon}(t,u^\epsilon_t,\mu^\epsilon_t) \big(\Phi^\epsilon_{(0,t)}(x)\big), \Psi_v^\epsilon(t,x)\Big\rangle\,d\mu_0(x)\,dt + o(\rho)
    \end{aligned}.
\end{equation}
Similarly, using the same notation of \eqref{eq: definition g sub u} with $g^\epsilon$ instead, it holds that

$$g^\epsilon \circ \mu^\epsilon_\rho - g^\epsilon \circ \mu^\epsilon = \bigg( t \mapsto \rho\int_{\R^d} \left(\inner{\nabla_\mu g^\epsilon(\mu^\epsilon_t)(\Phi^\epsilon_{(0,t)}(x)), \Psi^\epsilon_v(t,x)} \right) \, d \mu_0(x) \bigg) + o(\rho) \quad \text{ in } \quad C([0,T]).$$

Set
$$A_\epsilon := [J^\epsilon(u^\epsilon) - J(u^*) + \epsilon]^+, \qquad B_\epsilon := d_C(g^\epsilon \circ \mu^\epsilon)$$

By definition, $\mathscr{J}_\epsilon(u^\epsilon) = (A_\epsilon^2 + B_\epsilon^2)^{1/2}$, and we claim that $\mathscr{J}_\epsilon(u^\epsilon) > 0$. In fact, if not, $A_\epsilon = B_\epsilon = 0$. The equality $B_\epsilon = 0$ together with Proposition \ref{prop: LL on Wasserstein} implies that
$$g(\mu^\epsilon_t) \leq g^\epsilon(\mu^\epsilon_t) \leq 0, \qquad \forall t \in [0,T],$$
meaning that $u^\epsilon$ is feasible for \eqref{prob:P}. Since $A_\epsilon = 0$, again by Proposition \ref{prop: LL on Wasserstein} we deduce that
$$J(u^\epsilon) \leq J^\epsilon(u^\epsilon) \leq J(u^*) - \epsilon,$$
which is absurd because $u^\epsilon$ is feasible for \eqref{prob:P} and $u^*$ is a minimizer.
\smallskip

Since
\(\mathscr{J}_\epsilon(u^\epsilon)>0\), we have that
\(\alpha^\epsilon\ge0\). If $B_\epsilon>0$, then \(g^\epsilon \circ \mu^\epsilon\notin C\), so \(d_C\) is Gâteaux differentiable at \(g^\epsilon \circ \mu^\epsilon\) and
\[
\nabla d_C(g^\epsilon \circ \mu^\epsilon) \in \sP([0,T]).
\]

If $B_\epsilon = 0$, we set $\lambda^\epsilon := 0 \in \mcM_+([0,T])$. In any case, we have
\begin{equation}\label{eq:non-triviality-eps}
    |\alpha^\epsilon|^2 + |\lambda^\epsilon|_{\mathcal M([0,T])}^2 = 1.
\end{equation}
Let \( A_\rho := [J^{\epsilon}(u_\rho^\epsilon)-J(u^*)+\epsilon]^+\) and \( B_\rho:=d_C(g^\epsilon \circ \mu_\rho^\epsilon) \).  Since \(\mathscr{J}_\epsilon(u^\epsilon)>0\), the Euclidean norm is differentiable at $(A_\epsilon, B_\epsilon)$ and consequently,
\begin{align*}
\mathscr{J}_\epsilon(u_\rho^\epsilon) -\mathscr{J}_\epsilon(u^\epsilon) = \frac{A_\epsilon}{\mathscr{J}_\epsilon(u^\epsilon)} (A_\rho - A_\epsilon) + \frac{B_\epsilon}{\mathscr{J}_\epsilon(u^\epsilon)}(B_\rho-B_\epsilon) + o(\rho).
\end{align*}

If $A_\epsilon > 0$, then \(J^\epsilon(u^\epsilon)-J(u^*)+ \epsilon >0\) and for \(\rho>0\) sufficiently small,
\begin{equation}\label{eq: expansion A terms}
    \begin{aligned}
       \frac{A_\epsilon}{\mathscr{J}_\epsilon(u^\epsilon)} (A_\rho-A_\epsilon) &= \frac{A_\epsilon}{\mathscr{J}_\epsilon(u^\epsilon)}  \left( J^\epsilon(u_\rho^\epsilon) - J^\epsilon(u^\epsilon) \right)\\
        &= \frac{A_\epsilon}{\mathscr{J}_\epsilon(u^\epsilon)} \rho \bigg[\int_0^T \Big( L^{\epsilon}(t,v_\epsilon(t),\mu^\epsilon_t) - L^{\epsilon}(t,u^\epsilon_t,\mu^\epsilon_t) \Big)\,dt +  \int_{\R^d}\Big\langle
        \nabla_\mu\mcJ^\epsilon(\mu^\epsilon_T)
        \big(\Phi^\epsilon_{(0,T)}(x)\big),
        \Psi_v^\epsilon(T,x)
        \Big\rangle\, d\mu_0(x) \bigg]\\
        &+ \frac{A_\epsilon}{\mathscr{J}_\epsilon(u^\epsilon)} \rho\int_0^T\int_{\R^d} \Big\langle \nabla_\mu L^{\epsilon}(t,u^\epsilon_t,\mu^\epsilon_t) \big(\Phi^\epsilon_{(0,t)}(x)\big), \Psi_v^\epsilon(t,x)\Big\rangle\,d\mu_0(x)\,dt + o(\rho).
    \end{aligned}
\end{equation}
If $A_\epsilon =0$, \eqref{eq: expansion A terms} holds by definition. We now treat the term involving \(B_\rho-B_\epsilon\). If \(B_\epsilon>0\), then \(g^\epsilon \circ \mu^\epsilon\notin C\) and \(d_C\)
is Gâteaux differentiable at that point. Hence, setting
$$\lambda^\epsilon := \frac{B_\epsilon}{\mathscr{J}_\epsilon(u^\epsilon)} \nabla d_C(g^\epsilon \circ \mu^\epsilon)$$
we see that
\begin{align}\label{eq: expansion B terms}
    \frac{B_\epsilon}{\mathscr{J}_\epsilon(u^\epsilon)}\left( B_\rho - B_\epsilon \right) = \rho \int_{[0,T]} \bigg[ \int_{\R^d} \left(\inner{\nabla_\mu g^\epsilon(\mu^\epsilon_t)(x), \Psi^\epsilon_v(t,x)} \right) \, d \mu_0(x) \bigg] d \lambda^\epsilon(t) + o(\rho).
\end{align}
$$$$
If $B_\epsilon = 0$, \eqref{eq: expansion B terms} holds by definition since $\lambda^\epsilon = 0$. Combining \eqref{eq: expansion A terms}, \eqref{eq: expansion B terms} and \eqref{eq: almost taylor expansion ekeland principle} we deduce \eqref{eq: taylor expansion epsilon cost}.

\hfill

{\bf Step 4} (Convergences)

\hfill

In this step, we pass to the limit in \eqref{eq: almost taylor expansion ekeland principle} as $\epsilon \searrow 0$.

\paragraph{Preliminary convergences.}
\hfill

Let $v \in \mcU_{ad}$ and let $v_\epsilon$ be the truncation of \eqref{eq: truncation of controls}. By the dominated convergence theorem, we know that $d_{L^2}(v_\varepsilon,v)\longrightarrow0$. Let $\Psi^*_v$ the unique solution of

\begin{align*}
    \begin{cases}
        \partial_t\Psi(t,x) &= f[\mu_t^*]\bigl(t,v_t,\Phi^*_{(0,t)}(x)\bigr)-
        f[\mu_t^*]\bigl(t,u_t^*,\Phi^*_{(0,t)}(x)\bigr)\\
        &\quad+ D_x f[\mu_t^*]\bigl(t,u_t^*,\Phi^*_{(0,t)}(x)\bigr) \Psi(t,x)\\
        &\quad+
        \int_{\R^d}
        \Big\langle
        \nabla_\mu f[\mu_t^*]
        \bigl(t,u_t^*,\Phi^*_{(0,t)}(x)\bigr)
        \bigl(\Phi^*_{(0,t)}(y)\bigr),
        \Psi_v^*(t,y)
        \Big\rangle\,d\mu_0(y),\\
        \Psi(0,x)&=0.
    \end{cases}
\end{align*}
and let $\Psi^\epsilon_v$ the solution of \eqref{eq: linearized equation truncated}. By a classical Grönwall argument, using [A\ref{hyp:dynamic-non-smooth}] we have that $\Psi_v^\epsilon \to \Psi_v^*$ in $C([0,T];L^2_{\mu_0})$. Similarly, by \eqref{eq: stability flow} we have \(\Phi^\epsilon_{(0,\cdot)}
\longrightarrow
\Phi^*_{(0,\cdot)}\) in \(
C\bigl([0,T];L^2_{\mu_0}\bigr)
\).

For every \(s,t\in[0,T]\), set
\[\gamma_{s,t}^\epsilon := \big(\Phi^*_{(0,s)},\Phi^\epsilon_{(0,t)}\big)_\#\mu_0 \in \Gamma(\mu_s^*,\mu_t^\epsilon).
\]
In particular, if \(t_\epsilon\to t\), then
\begin{align}\label{eq: C2 goes to 0 for gamma epsilon t}
    \|\pi^2-\pi^1\|_{\gamma^\epsilon_{t,t_\epsilon}}
    \xrightarrow[\epsilon\searrow0]{}0.
\end{align}
Since the couplings \(\gamma^\epsilon_{s,t}\) are concentrated on a graph, for every $\Xi \in L^2_{\mu^\epsilon_t}$ it holds that
\begin{align}\label{eq: parallel transport induced by the flow}
    (\PP_{\gamma^\epsilon_{s,t}}\Xi)
    \big(\Phi^*_{(0,s)}(x)\big)
    =
    \Xi\big(\Phi^\epsilon_{(0,t)}(x)\big)
    \qquad \mu_0\text{-a.e. }x.
\end{align}

Let be multipliers $\alpha^\epsilon \geq 0, \lambda^\epsilon \in \mcM_+([0,T])$ of the previous step and choose a subsequence such that \( \alpha^\epsilon\to\alpha^*\) for some \(\alpha^*\in[0,1]\), and \(\lambda^\epsilon \rightharpoonup\lambda^*\) weakly-$\star$ for some \(\lambda^*\in\mathcal M_+([0,T])\).

\paragraph{Convergence of the running cost term.}

\hfill

Assumption \ref{hyp:cost-non-smooth}, Proposition \ref{prop: LL on Wasserstein}, the strong convergence $d_{L^2}(v_\varepsilon,v)\to0$,
$d_{L^2}(u^\varepsilon,u^*)\to0$,
and
$\sup_{t\in[0,T]}d_\mcW(\mu_t^\varepsilon,\mu_t^*)\to0$ together with the dominated convergence theorem imply that
$$\alpha^\epsilon
\int_0^T \Big(L^{\epsilon}(t,v_\epsilon(t),\mu_t^\epsilon) - L^{\epsilon}(t,u_t^\epsilon,\mu_t^\epsilon) \Big)\,dt \xrightarrow[\epsilon \searrow 0]{} \alpha^* \int_0^T \Big( L(t,v_t,\mu_t^*) - L(t,u_t^*,\mu_t^*) \Big)\,dt . $$

\paragraph{Convergence of the terminal cost term.}

\hfill

Let \(\Xi_\mcJ^\varepsilon
:=
\nabla_\mu\mcJ^\varepsilon(\mu_T^\varepsilon)\) and \(\widetilde\Xi_\mcJ^\epsilon
:=
\PP_{\gamma^\epsilon_{T,T}}\Xi_\mcJ^\epsilon\). By \eqref{ineq: bounds-pullback} and Proposition \ref{prop: LL on Wasserstein}, $\{\widetilde\Xi_\mcJ^\epsilon\}_\epsilon$ is bounded in $L^2_{\mu^*_T}$. Up to subsequence, we have that there is $\zeta^*_\mcJ \in L^2_{\mu^*_T}$ such that $\widetilde\Xi_\mcJ^\epsilon \rightharpoonup \zeta^*_\mcJ$ weakly in $L^2_{\mu^*_T}$. By Proposition \ref{prop: Clarke subdiff and Lasry Lions}, we deduce that
\begin{align}\label{eq: inclusion clarke subdiff terminal cost}
    \zeta_\mcJ^* \in \partial_C\mcJ(\mu_T^*).
\end{align}

Moreover, by \eqref{eq: parallel transport induced by the flow} we have that
$$\widetilde\Xi^\epsilon_\mcJ(\Phi^*_{(0,T)}(x)) = \Xi^\epsilon_\mcJ(\Phi^\epsilon_{(0,T)}(x)), \qquad \mu_0\text{-a.e. } x.$$
In particular,
$$\int_{\R^d}\Big\langle \nabla_\mu\mcJ^\epsilon(\mu_T^\epsilon)\big(\Phi^\epsilon_{(0,T)}(x)\big), \Psi_v^\epsilon(T,x)\Big\rangle\,d\mu_0(x)
= \int_{\R^d}\Big\langle \widetilde\Xi_\mcJ^\epsilon \big(\Phi^*_{(0,T)}(x)\big), \Psi_v^\epsilon(T,x)\Big\rangle\,d\mu_0(x).$$

By the strong convergence $\Psi^\epsilon_v(T, \cdot) \to \Psi_v^*(T,\cdot)$ and the weak convergence \(\widetilde\Xi_\mcJ^\varepsilon\circ\Phi^*_{(0,T)} \rightharpoonup \zeta_\mcJ^*\circ\Phi^*_{(0,T)}\) we deduce that
\begin{align*}
    &\alpha^\epsilon \int_{\R^d} \Big\langle \nabla_\mu\mcJ^\epsilon(\mu_T^\epsilon) \big(\Phi^\epsilon_{(0,T)}(x)\big),\Psi_v^\epsilon(T,x) \Big\rangle\,d\mu_0(x) \xrightarrow[\epsilon \searrow 0]{} \alpha^* \int_{\R^d} \Big\langle \zeta_\mcJ^* \big(\Phi^*_{(0,T)}(x)\big), \Psi_v^*(T,x) \Big\rangle\,d\mu_0(x).
\end{align*}

\paragraph{Derivative of the constraint term.}

\hfill

Let
\( \Xi^\epsilon_g(t) := \nabla_\mu g^\epsilon(\mu^\epsilon_t)\) and $\widetilde\Xi^\epsilon_g(t) := \PP_{\gamma^\epsilon_{t,t}} \Xi^\epsilon_g(t) \circ \Phi^*_{(0,t)} \in L^2_{\mu_0}$. By Proposition \ref{prop: LL on Wasserstein} there is \(C_g>0\) such that
\[
\|\widetilde\Xi^\epsilon_g(t)\|_{L^2_{\mu_0}}\le C_g
\qquad
\forall t\in[0,T].
\]
Let $B^w_{C_g}$ be the closed ball in $L^2_{\mu_0}$ with radius $C_g$ and define the family of measures
$$\Lambda^\epsilon_g := (\id, \widetilde\Xi^\epsilon_g)_\# \lambda^\epsilon \in \mcM_+([0,T] \times B^w_{C_g}).$$
Since $\{\lambda^\epsilon\}_\epsilon$ is bounded and $[0,T] \times B^w_{C_g}$ compact, there is $\Lambda^*_g$ such that $\Lambda^\epsilon_g \rightharpoonup^\star \Lambda^*_g$ in duality with $C([0,T] \times B^w_{C_g})$. Since $\lambda^\epsilon \rightharpoonup^\star \lambda^*$, the first marginal of $\Lambda^*_g$ is $\lambda^*$ and $\Lambda^*_g$ disintegrates as
$$\Lambda^*_g(dt, d\beta) = \lambda^*(dt) \sigma^g_t(d\beta)$$
where \((\sigma^g_t)_t\) is a \(\lambda^*\)-a.e. defined measurable family of
probability measures on \(B_{C_g}^w\). We claim that for $\lambda^*$-a.e. $t$,
\begin{align}\label{eq: support sigma clarke subdifferentiald constrained term}
    \supp \sigma^g_t \subset \{ \Xi \circ \Phi^*_{(0,t)} : \Xi \in \partial_C g(\mu^*_t) \}.
\end{align}
In fact, by \cite[Proposition 5.1.8]{ambrosioGradientFlowsMetric2005}, for $\lambda^*-$a.e. $t$ and $\sigma_t^g$-a.e. $\beta$ there is a sequence $(t_{\epsilon_n}, \beta_{\epsilon_n}) \in \supp \Lambda^{\epsilon_n}_g$ such that
$$t_{\epsilon_n} \to t, \qquad \beta_{\epsilon_n} \rightharpoonup \beta.$$
Since the measures $\Lambda^{\epsilon_n}_g$ are concentrated on a graph, we can assume without loss of generality that
$$\beta_{\epsilon_n} \circ \Phi^*_{(t,0)} = \widetilde\Xi^{\epsilon_n}_g(t_{\epsilon_n}) \circ \Phi^*_{(t,0)} = \PP_{\gamma^{\epsilon_n}_{t,t_{\epsilon_n}}} \nabla_\mu g^{\epsilon_n}(\mu^{\epsilon_n}_{t_{\epsilon_n}}) \in L^2_{\mu^*_t}.$$

Hence, Proposition \ref{prop: Clarke subdiff and Lasry Lions} we deduce that $\beta \circ \Phi^*_{(t,0)} \in \partial_C g(\mu^*_t)$ and \eqref{eq: support sigma clarke subdifferentiald constrained term} holds. Now, for $\lambda^*$-a.e. $t$ define
\begin{align}\label{eq: definition Xi tilde g}
    \widetilde\Xi^*_g(t):= \int_{B^w_{C_g}} \beta d \sigma^g_t(\beta).
\end{align}

By \eqref{eq: support sigma clarke subdifferentiald constrained term} and the convexity and closedness of the Clarke subdifferential, it holds that
$$\widetilde \Xi^*_g(t) \circ \Phi^*_{(t,0)} \in \partial_C g(\mu^*_t).$$

By the weak convergence of $\Lambda^\epsilon_g \rightharpoonup^\star \Lambda^*_g $ we deduce that for any $\Theta \in C([0,T]; L^2_{\mu_0})$,
\begin{align*}
    \int_{[0,T]} \left\langle \widetilde\Xi_g^*(t),\Theta(t) \right\rangle_{L^2_{\mu_0}}\,d \lambda^*(t) &= \int_{[0,T] \times B^w_{C_g}} \inner{\beta, \Theta(t)}_{L^2_{\mu_0}} d \Lambda^*_g(t,\beta)\\
    &= \lim_{\epsilon \searrow 0} \int_{[0,T] \times B^w_{C_g}}  \langle \beta,\Theta(t)\rangle_{L^2_{\mu_0}} \,d\Lambda_g^\epsilon(t,\beta)\\
    &=  \lim_{\epsilon \searrow 0} \int_{[0,T]} \langle \widetilde\Xi^\epsilon_g(t),\Theta(t)\rangle_{L^2_{\mu_0}} \,d \lambda^\epsilon(t).
\end{align*}

Setting $\Theta = \Psi^*_v$ and using the strong convergence $\Psi^{\epsilon}_v \to \Psi^*_v$ in $C([0,T]; L^2_{\mu_0})$ we deduce that

\begin{align*}
   \lim_{\epsilon \searrow 0} \int_{[0,T]} \int_{\R^d} \Big\langle \nabla_\mu g^\epsilon(\mu^\epsilon_t) \big(\Phi^\epsilon_{(0,t)}(x)\big), \Psi_v^\epsilon(t,x) \Big\rangle\,d\mu_0(x) \lambda^\epsilon(dt) &= \lim_{\epsilon \searrow 0} \int_{[0,T]} \inner{\widetilde\Xi^\epsilon_g(t), \Psi^\epsilon_v(t)}_{L^2_{\mu_0}} d \lambda^\epsilon(t)\\
   &=     \int_{[0,T]} \left\langle \widetilde\Xi_g^*(t),\Psi^*_v(t) \right\rangle_{L^2_{\mu_0}}\,d \lambda^*(t),
\end{align*}

\paragraph{Convergence of the gradients of the running cost.}
\hfill

Let \( \Xi_L^\epsilon(t) := \nabla_\mu L^{\epsilon}(t,u_t^\epsilon,\mu_t^\epsilon) \in L^2_{\mu_t^\epsilon}(\R^d;\R^d) \) and $\widetilde \Xi_L^\epsilon(t) := \PP_{\gamma^\epsilon_{t,t}} \Xi_L^\epsilon(t) \circ \Phi^*_{(0,t)} \in L^2_{\mu_0}$. By Proposition \ref{prop: LL on Wasserstein} and \eqref{ineq: bounds-pullback} there is $C_L>0$ such that
$$\forall \epsilon > 0, \qquad \sup_{t\in[0,T]}\norm{\widetilde\Xi_L^\epsilon(t)}_{L^2_{\mu_0}} \leq C_L.$$

Denote by 
\(B^{w}_{C_L}\) the closed ball in \(L^2_{\mu_0}\) of radius \(C_L\) endowed with the weak topology, and define the family of measures
\[
\Lambda_L^\epsilon:= (\id,u^\epsilon(\cdot),\widetilde\Xi_L^\epsilon(\cdot))_\#\mathscr L\quad \in\quad \mcM_+\big([0,T]\times U\times B_{C_L}^w\big)
\]

Up to subsequence, we can suppose that without loss of generality that $u^\epsilon \to u^*$ for almost every $t \in [0,T]$. In particular, $(\id,u^\varepsilon)_\#\mathscr L
\rightharpoonup
(\id,u^*)_\#\mathscr L.$ narrowly. Moreover, since $B^w_{C_L}$ is a compact metric space, the sequence $\left(\widetilde \Xi^\epsilon_L \right)_\# \mathscr{L}$ is relatively compact in $\sP_2(B^w_{C_L})$. Consequently, \cite[Lemma 5.2.2]{ambrosioGradientFlowsMetric2005} provides that there is a measure $\Lambda^*_L \in \mcM_+([0,T] \times U \times B^w_{C_L})$ such that
$$
\Lambda_L^\epsilon \rightharpoonup^\star \Lambda_L^* \qquad \text{ weakly in } \mcM_+([0,T] \times U \times B^w_{C_L}).
$$
Moreover, since $(\id, u^\epsilon)_\# \sL \rightharpoonup (\id, u^*)_\# \sL$, we have that $\Lambda^*_L$ disintegrates as
\begin{equation}\label{eq:proof-diisintegration-Lambda}
      \Lambda_L^*(dt,du,d\beta) = dt\,\delta_{u_t^*}(du)\,\sigma^L_t(d\beta),
\end{equation}
where \((\sigma^L_t)_{t\in[0,T]}\) is a measurable family of probability measures
on \(B_{C_L}^w\). As in the previous step, it holds that
\begin{align}\label{eq: sigma t supported on clarke subdifferential running cost term}
    \supp \sigma^L_t \subset \{ \Xi \circ \Phi^*_{(0,t)} : \Xi \in \partial_C L(t, u^*_t, \cdot)(\mu^*_t) \} \qquad \sL-\text{a.e. } t.
\end{align}

In fact, the disintegration formula
\eqref{eq:proof-diisintegration-Lambda}, for \(dt\)-a.e. \(t\) and for
\(\sigma_t^L\)-a.e. \(\beta\), the point \((t,u_t^*,\beta)\) belongs to
\(\supp\Lambda_L^*\). Hence, \cite[Proposition 5.1.8]{ambrosioGradientFlowsMetric2005} yield that there exist a subsequence $\epsilon_n \searrow 0$ and points $(t_{\epsilon_n}, u_{\epsilon_n}, \beta_{\epsilon_n}) \in \supp( \Lambda^{\epsilon_n}_L)$ such that
$$t_{\epsilon_n} \to t, \qquad u_{\epsilon_n} \to u^*_t, \qquad \beta_{\epsilon_n} \rightharpoonup \beta$$
Since $\Lambda^{\epsilon_n}_L$ is concentrated on a graph, we can suppose without loss of generality that
$$u_{\epsilon_n} = u^{\epsilon_n}(t_{\epsilon_n}), \qquad \beta_{\epsilon_n} = \widetilde\Xi^{\epsilon_n}_L(t_{\epsilon_n}), \qquad \mu_{\epsilon_n} := \mu^{\epsilon_n}_{t_n} .$$
By definition,
$$\widetilde \Xi_L^{\epsilon_n}(t_{\epsilon_n}) = \PP_{\gamma^{\epsilon_n}_{t_{\epsilon_n},t_{\epsilon_n}}} \Xi_L^{\epsilon_n}(t_{\epsilon_n}) \circ \Phi^*_{(0, t_{\epsilon_n})}, \qquad \Xi_L^{\epsilon_n}(t_{\epsilon_n}) = \nabla_\mu L^{\epsilon_n}(t_{\epsilon_n},u_{\epsilon_n},\mu_{\epsilon_n}).$$
On the one hand, if $L(t,u, \cdot)$ is convex along plans for every $(t,u)$, \eqref{eq: C2 goes to 0 for gamma epsilon t} and Proposition~\ref{prop:running-cost-LL-closedness} yield
$$\beta \circ\Phi^*_{(t,0)} \in \partial_C L(t, u^*_t, \cdot)(\mu^*_t).$$
On the other hand, if $L(t,u,\cdot) = L_1(t,u) + L_2(\cdot)$, it holds that
$$L^\epsilon(t,u,\mu) = L_1(t,u) + L_2^\epsilon(\mu).$$
In particular, \eqref{eq: C2 goes to 0 for gamma epsilon t} and Proposition \ref{prop: Clarke subdiff and Lasry Lions} yield that
$$\beta \circ\Phi^*_{(t,0)} \in \partial_C L(t, u^*_t, \cdot)(\mu^*_t).$$
In any case, \eqref{eq: sigma t supported on clarke subdifferential running cost term} holds. For $\sL$-a.e. $t$, define
\begin{align*}
    \widetilde \Xi^*_L(t) :=\int_{B^w_{C_L}} \beta d \sigma^L_t(\beta).
\end{align*}

Given that the Clarke subdifferential is closed and convex, by \eqref{eq: sigma t supported on clarke subdifferential running cost term} it holds that
\begin{align}\label{eq: clarke subdifferential inclusion running cost}
    \widetilde \Xi^*_L(t) \circ \Phi^*_{(t,0)} \in \partial_C L(t, u^*_t, \cdot)(\mu^*_t) \qquad \sL-\text{a.e. } t.
\end{align}
Moreover, by the weak convergence $\Lambda^\epsilon_L \rightharpoonup^\star \Lambda^*_L$, for any
\(\Theta\in C([0,T];L^2_{\mu_0}(\R^d;\R^d))\) it holds that
\begin{align*}
    \int_0^T \left\langle \widetilde\Xi_L^*(t),\Theta(t) \right\rangle_{L^2_{\mu_0}}\,dt& = \int \inner{\beta, \Theta(t) }_{L^2_{\mu_0}} d \Lambda^*_L(t,u,\beta)\\
    &= \lim_{\epsilon \searrow 0} \int \langle \beta,\Theta(t)\rangle_{L^2_{\mu_0}} \,d\Lambda_L^\epsilon(t,u,\beta)\\
    &= \lim_{\epsilon \searrow 0}  \int_0^T \left\langle \widetilde\Xi_L^\epsilon(t),\Theta(t) \right\rangle_{L^2_{\mu_0}}\,dt.
\end{align*}

Finally, the strong convergence $\Psi^\epsilon_v \to \Psi^*_v$ in $C([0,T]; L^2_{\mu_0})$ yields that
$$\lim_{\epsilon \searrow 0} \int_0^T \inner{\widetilde\Xi^\epsilon_L(t), \Psi^\epsilon_v(t)}_{L^2_{\mu_0}} \, dt = \int_0^T \inner{\widetilde \Xi^*_L(t), \Psi^*_v(t) }_{L^2_{\mu_0}} dt.$$

{\bf Step 5} (Maximization of the Hamiltonian).
\hfill

Collecting the previous convergences and using that
\(\sqrt{C \epsilon}T\to0\), we can pass to the limit in \eqref{eq: taylor expansion epsilon cost} and deduce that for \(v\in\mathcal U_{\rm ad}\),
\begin{equation}\label{eq:limit-variational-ineq}
    \begin{aligned}
    0\le\;&
    \alpha^*
    \int_{\R^d}
    \Big\langle
    \widetilde \Xi^*_\mcJ(x),
    \Psi_v^*(T,x)
    \Big\rangle\,d\mu_0(x)\\
    &+
    \alpha^*
    \int_0^T
    \Big(
    L(t,v_t,\mu_t^*)
    -
    L(t,u_t^*,\mu_t^*)
    \Big)\,dt\\
    &+
    \alpha^*
    \int_0^T\int_{\R^d}
    \Big\langle
    \widetilde\Xi^*_L(t)(x),
    \Psi_v^*(t,x)
    \Big\rangle\,d\mu_0(x)\,dt\\
    &+
    \int_{[0,T]}
    \int_{\R^d}
    \Big\langle
    \widetilde \Xi^*_g(t)(x),
    \Psi_v^*(t,x)
    \Big\rangle\,d\mu_0(x)\,
    \lambda^*(dt).
    \end{aligned}
\end{equation}
Set
$$\zeta^*_\mcJ =\widetilde \Xi^*_\mcJ \in \partial_C \mcJ(\mu^*_T), \qquad \zeta^*_L(t) :=\widetilde \Xi^*_L(t) \circ \Phi^*_{(t,0)} \in \partial_C L(t, u^*_t, \cdot)(\mu^*_t) \quad \sL-\text{a.e. } t,$$
$$\zeta^*_g(t) :=\widetilde\Xi^*_g(t) \circ \Phi^*_{(t,0)} \in \partial_C g(\mu^*_t) \qquad
\lambda^*\text{-a.e. }t\in[0,T],$$
and
\[
m^*(s):=
\int_{[0,s]}
\zeta_g^*(t)\big(\Phi^*_{(0,t)}(\cdot)\big)\,\lambda^*(dt).
\]
Let \(P^*\in BV([0,T];L^2_{\mu_0})\) be the solution of the backward
adjoint equation given by Proposition~\ref{prop:well-posedness-adjoint}:
\[
\left\{
\begin{aligned}
-dP^*(t,x)
&=
\Big[
D_x f[\mu_t^*]
\bigl(t,u_t^*,\Phi^*_{(0,t)}(x)\bigr)^\top P^*(t,x)
\\
&\quad+
\int_{\R^d}
\nabla_\mu f[\mu_t^*]
\bigl(t,u_t^*,\Phi^*_{(0,t)}(y)\bigr)
\bigl(\Phi^*_{(0,t)}(x)\bigr)^\top
P^*(t,y)\,d\mu_0(y)
\\
&\quad-
\alpha^*
\zeta_L^*(t)\bigl(\Phi^*_{(0,t)}(x)\bigr)
\Big]\,dt
-
dm^*(t,x),
\\
P^*(T,x)
&=
-\alpha^*
\zeta_\mcJ^*
\bigl(\Phi^*_{(0,T)}(x)\bigr).
\end{aligned}
\right.
\]
Using the duality identity between this adjoint equation and the linearized
equation for \(\Psi_v^*\), as in the smooth case, we obtain
\[
\begin{aligned}
&\alpha^*
\langle \zeta^*_\mcJ \circ \Phi^*_{(0,T)},\Psi_v^*(T)\rangle_{L^2_{\mu_0}}
+
\alpha^*
\int_0^T
\langle \zeta^*_L(t) \circ \Phi^*_{(0,t)},\Psi_v^*(t)\rangle_{L^2_{\mu_0}}\,dt\\
&\qquad
+
\int_{[0,T]}
\langle \zeta^*_g(t) \circ \Phi^*_{(0,t)},\Psi_v^*(t)\rangle_{L^2_{\mu_0}}\,\lambda^*(dt)\\
&=
-\int_0^T
\langle P^*(t), f[\mu_t^*]\bigl(t,v_t,\Phi^*_{(0,t)}(x)\bigr)
 -
 f[\mu_t^*]\bigl(t,u_t^*,\Phi^*_{(0,t)}(x)\bigr)\rangle_{L^2_{\mu_0}}\,dt .
\end{aligned}
\]
We now introduce the adjoint measure
\[
\gamma_t^*
:=
\big(\Phi^*_{(0,t)},P^*(t,\cdot)\big)_\#\mu_0
\in\mathscr P_2(\R^d\times\R^d).
\]
Then \(\pi^1_\#\gamma_t^*=\mu_t^*\), and~\eqref{eq:limit-variational-ineq} becomes
\begin{equation}\label{eq:integral-H-ineq}
  \int_0^T
\Big(
H(t,v_t,\gamma_t^*,\alpha^*)
-
H(t,u_t^*,\gamma_t^*,\alpha^*)
\Big)\,dt
\le0
\qquad
\forall v\in\mathcal U_{\rm ad},
\end{equation}
where 
\[
H(t,u,\gamma,\alpha)
:=
\int_{\R^{2d}}
\langle p,f[\pi^1_\#\gamma](t,u,x)\rangle\,d\gamma(x,p)
-
\alpha\,L(t,u,\pi^1_\#\gamma),
\]
Finally, by the continuity of \(u\mapsto H(t,u,\gamma_t^*,\alpha^*)\) and the classical localization argument,
the inequality extends to every \(u\in U\). Therefore
\[
H(t,u_t^*,\gamma_t^*,\alpha^*)
=
\max_{u\in U}
H(t,u,\gamma_t^*,\alpha^*)
\qquad
\text{for a.e. }t\in[0,T]. \]

{\bf Step 6} (Non-triviality of the multiplier and complementary slackness).
\hfill

To obtain the non-triviality condition, suppose that \(\alpha^*=0\) and \(\lambda^*=0\). Since
\(\lambda^\epsilon\in\mathcal M_+([0,T])\) and
\(\lambda^\epsilon\rightharpoonup^\star0\), testing against the constant
function \(1\) yields
\[
\|\lambda^\epsilon\|_{\mathrm{TV}}
=
\lambda^\epsilon([0,T])
\longrightarrow0.
\]
The dual norm associated with the equivalent norm chosen on \(C([0,T])\) is
equivalent to the total variation norm, and therefore
\[
|\lambda^\epsilon|_{\mathcal M([0,T])}\longrightarrow0.
\]
Together with~\eqref{eq:non-triviality-eps}, this implies \(\alpha^\epsilon\to1\), contradicting
\(\alpha^\epsilon\to\alpha^*=0\). Hence
\[
(\alpha^*,\lambda^*)\neq(0,0).
\]

The complementary slackness condition follows by using the convexity of the set $C$, together with the uniform convergence of $g^\epsilon \circ \mu^\epsilon$ to $g \circ \mu^*$ and the weak-$\star$ convergence of $\lambda^\epsilon \rightharpoonup\lambda^*$, c.f. \cite[p.22 and p.23]{urrea2026pontryagin}

\appendix
\titleformat{\section}
  {\normalfont \large\bfseries}
  {Appendix \Alph{section}:}
  {0.5 em}
  {}

\section{Further structural properties and proofs}\label{append: Additional properties Clarke subdiff}

This appendix collects further structural properties of the Clarke subdifferential and provides the proof of Theorem \ref{thm: clarke subdiff of wass dist}. Its purpose is threefold. First, we relate the Clarke subdifferential to the strong Fréchet subdifferential of \cite{ambrosioGradientFlowsMetric2005}. Second, we establish the basic calculus rules inherited from the classical Clarke theory. Finally, we develop the geometry of tangent plans needed for the proof of the explicit formula for the Clarke subdifferential of the squared Wasserstein distance.\\

As a byproduct, we also obtain a sharpened formulation of the strong Fréchet superdifferential of $d_\mcW^2(\cdot,\sigma)$; see Proposition \ref{prop: strong frechet superdifferential of wass dist}.

\subsection{Relations between Fréchet and Clarke subdifferential}

In this section we provide a relation between the Clarke subdifferential and the strong Fréchet subdifferential introduced in \cite{ambrosioGradientFlowsMetric2005}.

\begin{definition}
    Let $\mcG: \sP_2(\R^d) \to \R$ be a mapping and let $\mu \in \sP_2(\R^d)$. The sets of Fréchet strong subdifferentials of $\mcG$ at $\mu$ are defined as follows:
    \begin{enumerate}[i)]
        \item The set $\boldsymbol{\partial}_S \mcG(\mu)$ is the set of all the measures $\xi \in \sP_2(T\R^d)_\mu$ such that for any $\boldsymbol{\mu} \in \sP_2(T^2 \R^d)$ satisfying $\pi^{1,2}_\# \boldsymbol{\mu} = \xi$, $\pi^3_\# \boldsymbol{\mu} = \nu$,

        $$\mcG(\nu) \geq \mcG(\mu) + \int \inner{x_2, x_3 - x_1} d \boldsymbol{\mu}(x_1,x_2,x_3) + o( \norm{\pi^3 - \pi^1}_{\boldsymbol{\mu}})$$

        \item The set $\partial_S \mcG(\mu)$ is the set of all the maps $\Xi \in L^2_\mu$ such that for any $\gamma \in \Gamma(\mu, \nu)$ it holds that
        $$\mcG(\nu) \geq \mcG(\mu) + \int \inner{\Xi(x), y-x} d \gamma(x,y) + o( \norm{\pi^2 - \pi^1}_\gamma).$$
    \end{enumerate}
    The set of Fréchet strong superdifferentials $\boldsymbol{\partial}^S \mcG(\mu)$ and $\partial^S \mcG(\mu)$ are defined with $\leq$ instead.
\end{definition}

It can be proven (see e.g. \cite[Proposition 5.21]{bertucci2025tangent}) that
\begin{align}\label{eq: frechet subdiff and lifting}
    \boldsymbol{\partial}_S \mcG(\mu) = \left\{(X,Z)_\# \mbP : \, X \sim \mu, \, Z \in \partial_F \hat{\mcG}(X) \right\}
\end{align}

Similarly, a disintegration argument shows that
\begin{align}\label{eq: frechet subdiff by plans and barycenter}
    \partial_S \mcG(\mu) = \mcB[ \boldsymbol{\partial}_S \mcG(\mu)].
\end{align}

Both \eqref{eq: frechet subdiff and lifting} and \eqref{eq: frechet subdiff by plans and barycenter} also hold for the set of strong Fréchet superdifferentials.

\begin{proposition}\label{prop: frechet and clarke subdiff}
    Let $\mcG: \sP_2(\R^d) \to \R$ be locally Lipschitz around $\mu$. Then,
    $$\partial_S \mcG(\mu) \subset \partial_C \mcG(\mu).$$
\end{proposition}

\subsection{Calculus rules for the Clarke subdifferential}

Most of the computational rules of the classical Clarke subdifferential are inherited by the Clarke subdifferential in $\sP_2(\R^d)$.

\begin{theorem}\label{thm: computational properties Clarke subdiff}
    Let $\mcG$, $\mcG_1,...., \mcG_n$ be locally Lipschitz around $\mu \in \sP_2(\R^d)$.
    \begin{enumerate}[i)]
        \item For every $\beta \in L^2_\mu$, $D\mcG(\mu)(-\beta) = D (-\mcG)(\mu)(\beta)$.
        
        \item Define $\mcF:= \max_{i = 1,...,n} \mcG_i$ and let $I_\mu := \{ i : \mcF(\mu) = \mcG_i(\mu)\}$. Then,
        $$\partial_C \mcF(\mu) \subset \operatorname{conv}\{ \partial_C \mcG_i(\mu) : i \in I_\mu \}.$$

        \item It holds that
        $$\partial_C (\mcG_1 + .... + \mcG_n)(\mu) \subset \partial_C \mcG_1(\mu) + ... + \partial_C \mcG_n(\mu).$$
    \end{enumerate}
\end{theorem}

\subsection{Geometric ingredients for the squared Wasserstein distance}\label{sec: clarke subdiff of squared wass dist}

This subsection provides the geometric ingredients underlying the proof of Theorem \ref{thm: clarke subdiff of wass dist}. We first recall the geometry of tangent and solenoidal plans, then establish a closure property for convex hulls, and finally derive the corresponding description of the strong Fréchet superdifferential.

\subsubsection{Tangent and solenoidal plans}

Let $T^2\R^d := \R^d \times (\R^d)^2$ and let $\pi^x, \pi^v, \pi^w: T^2\R^d \to \R^d$ be the projections onto the first, second and third coordinate respectively. On $\sP_2(T \R^d)_\mu$ introduce the distance
$$d_\mu^2(\zeta, \xi) := \inf_{\eta \in \Gamma_\mu(\zeta,\xi)} \int_{T^2\R^d} |v-w|^2 d \eta(x,v,w),$$
where
$$\Gamma_\mu(\zeta,\xi) := \{ \eta \in \sP_2(T^2\R^d) : (\pi^x, \pi^v)_\# \eta = \zeta, \quad (\pi^x, \pi^w)_\# \eta = \xi \}.$$
It is known that the infimum above is always reached, and we denote by $\Gamma^o_\mu(\zeta,\xi)$ the set of optimal plans between $\zeta$ and $\xi$. It is known that $(\sP_2(T\R^d)_\mu, d_\mu)$ is a complete metric space. For every $\lambda \in \R$ and $\xi, \zeta \in \sP_2(T\R^d)_\mu$, define
$$\lambda \xi := (\pi^x, \lambda \pi^v)_\# \xi, \qquad \norm{\xi}_\mu^2 := \int_{T\R^d} |v|^2 d \xi(x,v), \qquad \inner{\zeta,\xi}_\mu := \frac{1}{2}\big(\norm{\zeta}_\mu^2 + \norm{\xi}_\mu^2 - d_\mu^2(\zeta, \xi) \big).$$
By definition,
$$ \inner{\xi,\zeta}_\mu = \inner{\zeta, \xi}_\mu = \sup_{ \eta \in \Gamma_\mu(\xi, \zeta)} \int_{T^2 \R^d} \inner{v,w} d \eta(x,v,w).$$

The zero element $\mathbf{0}_\mu \in \sP_2(T\R^d)_\mu$ is defined as the measure $(\id, 0)_\# \mu$. In particular, for every $\xi \in \sP_2(T\R^d)_\mu$, $0 \xi = \mathbf{0}_\mu$ and $\norm{\xi}_\mu = 0$ if and only if $\xi = \mathbf{0}_\mu$.\\

The set of tangent plans $\bfTan_\mu \sP_2(\R^d) \subset \sP_2(T\R^d)_\mu$ corresponds to the set
$$\bfTan_\mu \sP_2(\R^d) := \overline{\bigg\{ \zeta \in \sP_2(T\R^d)_\mu : \exists t > 0, \quad (\pi^x, \pi^x + t \pi^v)_\# \zeta \in \Gamma_o\left(\mu, (\pi^x + t \pi^v)_\# \zeta\right)  \bigg\} }^{d_\mu}.$$
In \cite[Proposition 4.29]{giglithesis} it is shown that
\begin{align}\label{eq: scalar multiplication and tangent plans}
    \forall \lambda \in \R, \forall \zeta \in \bfTan_\mu \sP_2(\R^d), \qquad \lambda \zeta \in \bfTan_\mu \sP_2(\R^d)
\end{align}
The set of solenoidal plans $\mathbf{Sol}_\mu \sP_2(\R^d)$, introduced in \cite{aussedat2025structure}, corresponds to the set
$$\mathbf{Sol}_\mu \sP_2(\R^d) := \left\{ \zeta \in \sP_2(T\R^d)_\mu :  \forall \xi \in \bfTan_\mu \sP_2(\R^d), \quad \inner{\zeta, \xi}_\mu = 0 \right\}.$$
In particular, \eqref{eq: scalar multiplication and tangent plans} yields that
\begin{align}\label{eq: scalar multiplication and solenoidal plans}
    \forall \lambda \in \R, \forall \zeta \in \mathbf{Sol}_\mu \sP_2(\R^d), \qquad \lambda \zeta \in \mathbf{Sol}_\mu \sP_2(\R^d).
\end{align}
By \cite[Lemma 2.3]{aussedat2025structure}, it also holds that
\begin{align}\label{eq: derivative equal 0 for solenoidal mappings}
    \forall \zeta \in \mathbf{Sol}_\mu \sP_2(\R^d), \forall \sigma \in \sP_2(\R^d), \qquad \lim_{t \searrow 0} \frac{d_\mcW^2( (\pi^x + t \pi^v)_\# \zeta), \sigma) - d_\mcW^2(\mu, \sigma)}{t} = 0.
\end{align}

It can be proven (c.f. \cite[Proposition 4.30]{giglithesis} and \cite[Lemma 4.3]{aussedat2025structure}) that for every element $\xi \in \sP_2(T\R^d)_\mu$ there is a unique element $\pi^\mu_T\xi \in \bfTan_\mu\sP_2(\R^d)$ (resp. $\pi^\mu_S \xi \in \mathbf{Sol}_\mu \sP_2(\R^d))$ minimizing the distance from $\xi$ to $\bfTan_\mu \sP_2(\R^d)$ (resp. to $\mathbf{Sol}_\mu \sP_2(\R^d)$). Moreover, \cite[Theorem 4.2]{aussedat2025structure} 
yields the following.

\begin{theorem}[Helmholtz-Hodge decomposition]\label{thm: helmholtz hodge decomposition}
    For every $\xi \in \sP_2(T\R^d)_\mu$ there is a mapping $T_\xi: T\R^d \to T \R^d$ with $\pi^x \circ T_\xi = \pi^x$ such that
    $$\Gamma^o_\mu(\xi, \pi^\mu_T\xi) = \{ (\pi^x, \pi^v, T_\xi)_\# \xi \}, \qquad \Gamma^o_\mu(\xi, \pi^\mu_S\xi) = \{ (\pi^x, \pi^v, \pi^v - T_\xi)_\# \xi \}.$$
    Moreover,
    $$\forall \zeta \in \mathbf{Sol}_\mu \sP_2(\R^d), \qquad \inner{\xi, \zeta}_\mu = \inner{\pi^\mu_S \xi, \zeta}_\mu.$$
\end{theorem}

\begin{remark}\label{remark: orthogonality of solenoidal and tangent components}
    A strong consequence of the above theorem is that it allows one to characterize the elements on $\bfTan_\mu \sP_2(\R^d)$ by orthogonality. More specifically,
    \begin{align*}
        \xi \in \bfTan_\mu \sP_2(\R^d) \quad \iff \quad \pi^\mu_S \xi = \boldsymbol{0}_\mu \quad \iff \quad \forall \zeta \in \mathbf{Sol}_\mu \sP_2(\R^d), \quad \inner{\xi, \zeta}_\mu = 0.
    \end{align*}
\end{remark}

\subsubsection{Convex hulls in the space of tangent plans}

\begin{definition}[Convex sets and convex hull on $\sP_2(T\R^d)_\mu$]
    A set $D \subset\sP_2(T \R^d)_\mu$ is said to be convex if for every $\zeta, \xi \in D$, every $t \in [0,1]$ and any $\eta \in \Gamma_\mu(\zeta, \xi)$ it holds that $(\pi^x, (1-t) \pi^v + t \pi^w)_\# \eta \in D$. For a set $C \subset \sP_2(T\R^d)_\mu$ the convex hull $\operatorname{conv}(C)$ of $C$ is the smallest convex set containing it. 
\end{definition}

Adapting the classical argument on vector spaces (see \cite[Remark 5.1.10]{aussedatphdthesis}) it can be deduced that
\begin{align}\label{eq: formula convex hull}
    \operatorname{conv}(C) = \bigcup_{n \in \N} C_n, \quad C_0 := C, \, C_{n+1} := \{ (\pi^x, (1-t) \pi^v + t \pi^w)_\# \eta : \eta \in \Gamma_\mu(\zeta, \xi), \zeta,\xi \in C_n, t \in [0,1] \}.
\end{align}

In \cite{aussedatphdthesis} the closed convex hull of a set $C \subset \sP_2(T\R^d)_\mu$ is defined as the smallest closed (w.r.t. $d_\mu$) convex set containing $C$. Since the geometry of $\sP_2(T\R^d)_\mu$ is non-flat, it is not immediate that the metric closure of $\operatorname{conv}(C)$ remains convex. The following proposition establishes precisely this fact.

\begin{proposition}\label{prop: closed convex hull}
    Let $C \subset \sP_2(T\R^d)_\mu$. The closed convex hull of $C$ is equal to the closure w.r.t. $d_\mu$ of the convex hull of $C$.
\end{proposition}

\begin{proof}
    Let $D\subset \sP_2(T\R^d)_\mu$ denote the closed convex hull of $C$. Clearly $\overline{\operatorname{conv}(C)}^{d_\mu} \subset D$. To prove the other inclusion, it suffices to prove that $\overline{\operatorname{conv}(C)}^{d_\mu}$ is convex. Let $\zeta, \xi \in \overline{\operatorname{conv}(C)}^{d_\mu}$, $\eta \in \Gamma_\mu(\zeta, \xi)$ and $t \in [0,1]$. By definition, there are elements $\zeta_n, \xi_n \in \operatorname{conv}(C)$ with $\zeta_n \to \zeta$ and $\xi_n \to \xi$ in $d_\mu$. Let $\eta^1_n \in \Gamma^o_\mu(\zeta, \zeta_n)$ and $\eta^2_n \in \Gamma^o_\mu(\xi, \xi_n)$. Disintegrating $\eta$  and $\eta^1_n$ w.r.t. to $\zeta$, using the gluing lemma \cite[Lemma 5.3.2]{ambrosioGradientFlowsMetric2005}, disintegrating the resulting measure and $\eta^2_n$ w.r.t. $\xi$ and using the gluing lemma again, we can construct a measure $\bar{\eta}_n \in \sP_2(T^4\R^d)$ such that
    $$(\pi^x, \pi^v, \pi^w)_\# \bar{\eta}_n = \eta, \quad (\pi^x, \pi^v, \pi^u)_\# \bar{\eta}_n = \eta^1_n, \quad (\pi^x, \pi^w, \pi^z)_\# \bar{\eta}_n = \eta^2_n,$$
    where the elements in $T^4\R^d$ are denoted by $(x,v,w,u,z)$. Define $\eta_n  := (\pi^x, \pi^u, \pi^z)_\# \bar{\eta}_n\in \Gamma_\mu(\zeta_n, \xi_n)$. Since $\operatorname{conv}(C)$ is convex, we have that
    $$(\pi^x, (1-t) \pi^v + t \pi^w)_\# \eta_n \in \operatorname{conv}(C).$$
    Moreover,
    \begin{align*}
        &d_\mu^2 \big((\pi^x, (1-t) \pi^v + t \pi^w)_\# \eta, (\pi^x, (1-t) \pi^v + t \pi^w)_\# \eta_n,  \big)\\
        &\leq \int_{T^4\R^d} (1-t)^2|v-u|^2 + t^2 |w-z|^2 d \bar{\eta}_n(x,v,w,u,z)\\
        &\leq (1-t)^2 d_\mu^2(\zeta, \zeta_n) + t^2 d_\mu^2(\xi, \xi_n) \xrightarrow[n \to +\infty]{}0.
    \end{align*}
    Hence, $(\pi^x, (1-t) \pi^v + t \pi^w)_\# \eta \in \overline{\operatorname{conv}(C)}^{d_\mu}.$
\end{proof}

\subsubsection{Strong Fréchet superdifferential and proof of Theorem \ref{thm: clarke subdiff of wass dist}}

The preceding closure result of Proposition \ref{prop: closed convex hull} allows us to formulate the strong Fréchet superdifferential of the squared Wasserstein distance in the following
form. This formulation sharpens the characterization in \cite[Theorem 5.1.12]{aussedatphdthesis} and is the form needed in the proof of Theorem \ref{thm: clarke subdiff of wass dist}.

\begin{proposition}\label{prop: strong frechet superdifferential of wass dist}
    Let $\mu, \sigma \in \sP_2(\R^d)$. Then,
    \begin{align}\label{eq: strong frechet superdifferential of wass dist}
        \boldsymbol{\partial}^S d_\mcW^2(\cdot, \sigma)(\mu) = \overline{\operatorname{conv}\bigg( \left\{(\pi^x, -2(\pi^y - \pi^x))_\# \gamma : \gamma \in \Gamma_o(\mu, \sigma) \right\} \bigg)}^{d_\mu}.
    \end{align}
\end{proposition}

\begin{proof}
    Let $\gamma \in \Gamma_o(\mu, \sigma)$. By \cite[Theorem 10.2.2]{ambrosioGradientFlowsMetric2005}, $\xi :=(\pi^x, - 2(\pi^y - \pi^x))_\# \gamma \in \boldsymbol{\partial}^S d_\mcW^2(\cdot, \sigma)(\mu)$ and for every $\boldsymbol{\mu} \in \sP_2(T\R^d \times \R^d)$ with $\pi^{1,2}_\# \boldsymbol{\mu} = \xi$ it holds that
    \begin{align*}
        d_\mcW^2(\pi^3_\# \boldsymbol{\mu}, \sigma) \leq d_\mcW^2(\mu, \sigma) + \int_{T\R^d \times \R^d} \inner{x_2, x_3 - x_1} d \boldsymbol{\mu}(x_1,x_2,x_3) + 2 \norm{\pi^3 - \pi^1}^2_{\boldsymbol{\mu}}.
    \end{align*}
    As a direct consequence of the above inequality and \eqref{eq: formula convex hull}, we see that
    $$\operatorname{conv}\bigg( \left\{(\pi^x, -2(\pi^y - \pi^x))_\# \gamma : \gamma \in \Gamma_o(\mu, \sigma) \right\} \bigg) \subset \boldsymbol{\partial}^S d_\mcW^2(\cdot, \sigma)(\mu).$$
    Moreover, the same approximation argument used in Proposition \ref{prop: closed convex hull} shows that we can pass to the limit and obtain the corresponding inclusion. To prove the other inclusion we recall \cite[Theorem 5.1.12]{aussedatphdthesis}. In fact, it is proven there that if a measure $\xi \in \sP_2(T\R^d)_\mu$ belongs to $\bfTan_\mu \sP_2(\R^d)$ and satisfies the condition
    \begin{equation}\label{eq: weaker version of frechet superdiff}
        \begin{aligned}
            \forall \nu \in &\sP_2(\R^d), \forall \boldsymbol{\mu} \in \sP_2(T\R^d \times \R^d), \quad (\pi^1, \pi^2)_\#\boldsymbol{\mu} = \xi, \quad (\pi^1, \pi^3)_\# \boldsymbol{\mu} \in \Gamma_o(\mu, \nu),\\
            &d_\mcW^2(\nu, \sigma) \leq d_\mcW^2(\mu, \sigma) + \int_{T\R^d \times \R^d} \inner{x_2, x_3 - x_1} d \boldsymbol{\mu}(x_1,x_2,x_3) + o(d_\mcW(\nu, \mu))
        \end{aligned}
    \end{equation}
    then $\xi$ belongs to the closed convex hull of the set $\left\{(\pi^x, -2(\pi^y - \pi^x))_\# \gamma : \gamma \in \Gamma_o(\mu, \sigma) \right\}$. By Proposition \ref{prop: closed convex hull} this set corresponds to the closure in the $d_\mu$-topology of the convex hull of the set $ \left\{(\pi^x, -2(\pi^y - \pi^x))_\# \gamma : \gamma \in \Gamma_o(\mu, \sigma) \right\}$. Moreover, we see that the condition \eqref{eq: weaker version of frechet superdiff} is clearly weaker than the condition $\xi \in \boldsymbol{\partial}^S d_\mcW^2(\cdot, \sigma)(\mu)$. Hence, the proof follows if we prove that $\boldsymbol{\partial}^S d_\mcW^2(\cdot, \sigma)(\mu) \subset \bfTan_\mu \sP_2(\R^d)$. To prove it, let $\xi \in \boldsymbol{\partial}^S d_\mcW^2(\cdot, \sigma)(\mu)$, $\zeta \in \mathbf{Sol}_\mu \sP_2(\R^d)$ and let $\eta \in \Gamma_\mu(\xi, \zeta)$. Applying the definition of strong Fréchet superdifferential on the measure $\boldsymbol{\mu} := (\pi^x, \pi^v, \pi^x + t \pi^w)_\# \eta$, dividing by $t$, letting $t \searrow 0$ and using \eqref{eq: derivative equal 0 for solenoidal mappings}, we deduce that
    \begin{align*}
        \forall \eta \in \Gamma_\mu(\xi, \zeta), \qquad  0 = \lim_{t \searrow 0} \frac{d_\mcW^2((\pi^x + t \pi^v)_\# \zeta, \sigma) - d_\mcW^2(\mu, \sigma)}{t} \leq \int \inner{v,w} d \eta(x,v,w).
    \end{align*}
    Taking $-\zeta$ instead of $\zeta$ and using \eqref{eq: scalar multiplication and solenoidal plans} we deduce that
    $$\forall \zeta \in \mathbf{Sol}_\mu \sP_2(\R^d), \qquad \inner{\xi, \zeta}_\mu = 0.$$
    Remark \ref{remark: orthogonality of solenoidal and tangent components} yields that $\xi \in \bfTan_\mu \sP_2(\R^d)$.
\end{proof}

\begin{proof}[Proof of Theorem \ref{thm: clarke subdiff of wass dist}]
    Define $\mcG(\mu) := d_\mcW^2(\mu, \sigma)$ and let $\hat{\mcG}$ be its lift. By \cite[Lemma 2.1]{bertucciApproximationSquaredWasserstein2024}, for every $X \in L^2_\mbP$ we have
    $$\hat{\mcG}(X) = |X|^2 + M_2^2(\sigma) - 2 \sup_{ Y \sim \sigma} \inner{X,Y}.$$
    We see that the mapping $X \mapsto \sup_{Y \sim \sigma} \inner{X,Y}$ is convex. Denoting by $\partial_{\operatorname{conv}}, \partial_F$ (resp. $\partial^{\operatorname{conc}}, \partial^F$) the usual convex and Fréchet subdifferential (resp. concave and Fréchet superdifferential) on $L^2_\mbP$ and using \cite[Theorem 10.8]{clarke2013functional}, we obtain
    \begin{align*}
        \partial_C \hat{\mcG}(X) = 2X + \partial^{\operatorname{conc}}\bigg( -2\sup_{Y \sim \sigma} \inner{\cdot, Y} \bigg)(X) = \partial^F \bigg( |\cdot|^2- 2 \sup_{Y \sim \sigma} \inner{\cdot, Y} \bigg)(X) = \partial^F \hat{\mcG}(X).
    \end{align*}
    Using Propositions \ref{prop: clarke subdifferential and lifting}, \ref{prop: strong frechet superdifferential of wass dist} and \eqref{eq: frechet subdiff and lifting} we deduce that
    \begin{align*}
        \partial_C \mcG(\mu) = \bigcup_{X \sim \mu} \Bar_X[\partial_C \hat{\mcG}(X)] = \mcB \bigg[  \left\{ (X, Z)_\# \mbP : X \sim \mu, Z \in \partial^F \hat{\mcG}(X) \right\}  \bigg] &= \mcB\bigg[ \boldsymbol{\partial}^S \mcG(\mu) \bigg]
    \end{align*}
    On the one hand, since the mapping $\mcB: \left( \sP_2(T \R^d)_\mu, d_\mu \right) \to \left(L^2_\mu, \norm{\cdot}_\mu \right)$ is continuous, using the explicit formula \eqref{eq: formula convex hull} together with the linearity of the barycenter projection w.r.t. the interpolations appearing on the mentioned formula and Proposition \ref{prop: strong frechet superdifferential of wass dist},
    \begin{align*}
        \mcB[ \boldsymbol{\partial}^S \mcG(\mu)]&\subset  \overline{\operatorname{conv}\bigg( \left\{2 \id- 2\mcB[\gamma] : \gamma \in \Gamma_o(\mu, \sigma) \right\} \bigg)}^{L^2_\mu}.
    \end{align*}
    On the other hand, Propositions \ref{prop: frechet and clarke subdiff}, \ref{prop: strong frechet superdifferential of wass dist} and \eqref{eq: frechet subdiff by plans and barycenter} yield
    $$\{ 2 \id- 2\mcB[\gamma] : \gamma \in \Gamma_o(\mu, \sigma) \} \subset  \mcB[\boldsymbol{\partial}^S \mcG(\mu)] = \partial_C \mcG(\mu).$$
    The conclusion follows by the fact that $\partial_C \mcG(\mu)$ is convex and closed.
\end{proof}

\section{Diffusive perturbation}\label{sec: diffusive perturbation}

\begin{proof}[Proof of Proposition \ref{prop:diffuse-linearization}]
    We follow the proof of~\cite[Proposition~5.1]{urrea2026pontryagin}. 
    
    %
    %
    Set
    $$
        \Delta f(t,x):=
        f[\bar\mu_t]
        \bigl(t,u_t,\Phi^{\bar u}_{(0,t)}(x)\bigr)
        -
        f[\bar\mu_t]\bigl(t,\bar u_t,\Phi^{\bar u}_{(0,t)}(x)\bigr),
    $$
    and
    $$
        \Delta L^\epsilon(t)
        :=
        L^\epsilon(t,u_t,\bar\mu_t)
        -
        L^\epsilon(t,\bar u_t,\bar\mu_t).
    $$
    
    [A\ref{hyp:dynamic-non-smooth}.i], together with the measurability of the controls and the flow, implies that $(t, x) \mapsto$ $\Delta f(t, x)$ is jointly measurable. Consequently, $t \mapsto \Delta f(t) \in L^2_{\mu_0}$ is weakly $\sL$-measurable. Since $L^2_{\mu_0}$ is separable, Pettis' measurability theorem (see~\cite[Theorem 2]{diestel1977vector}) yields the $\sL-$measurability of $t \mapsto \Delta f(t)$.
    Moreover, [A\ref{hyp:dynamic-non-smooth}.ii] yields
    \[
    \|\Delta f(t)\|_{L^2_{\mu_0}}
    \leq
    \ell_f(t)\,d_U(u(t),\bar u(t))
    \quad\text{for a.e. }t\in[0,T].
    \]
    Since both factors on the right-hand side belong to $L^2(0, T)$, the norm of $\Delta f$ is integrable. Together with the strong measurability established above, this gives
    \(
    \Delta f \in L^1(0, T ; L^2_{\mu_0}) 
    \).
    For every \(\rho>0\), choose a simple function
    \(\Delta f_\rho\in L^1(0,T; L^2_{\mu_0})\) such that
    \begin{equation}\label{eq:approx-delta-f}
    \|\Delta f-\Delta f_\rho\|_{L^1(0,T; L^2_{\mu_0} )}
    \leq\rho^2.
    \end{equation}
    Since \(\Delta f_\rho\) takes finitely many values, its range is contained in
    a finite-dimensional subspace \(H_\rho\subset L^2_{\mu_0}\).
    
    Choose a partition
    \[
    0=t_0<t_1<\cdots<t_{m_\rho}=T
    \]
    such that
    \begin{equation}\label{eq:partition-delta-f-metric}
\int_{t_{j-1}}^{t_j}
\|\Delta f_\rho(t)\|_{L^2_{\mu_0}}\,dt
\le\rho^2,
    \qquad j=1,\ldots,m_\rho.
    \end{equation}
    On each \(I_j=[t_{j-1},t_j]\), apply Lyapunov's convexity theorem to
    \[
    t\longmapsto
    \bigl(1,\Delta L^\epsilon(t),\Delta f_\rho(t)\bigr)
    \in\mathbb R^2\times H_\rho.
    \]
    This yields a measurable set \(E_{\rho,j}\subset I_j\) such that
    \[
    \begin{aligned}
    \mathcal L(E_{\rho,j})
    &=\rho\mathcal L(I_j),\\
    \int_{E_{\rho,j}}\Delta L^\epsilon(t)\,dt
    &=\rho\int_{I_j}\Delta L^\epsilon(t)\,dt,\\
    \int_{E_{\rho,j}}\Delta f_\rho(t)\,dt
    &=\rho\int_{I_j}\Delta f_\rho(t)\,dt.
    \end{aligned}
    \]
    Setting \(E_\rho:=\bigcup_jE_{\rho,j}\), we obtain
    \[
    \mathcal L(E_\rho)=\rho T
    \]
    and~\eqref{eq:Lyap-L}. Moreover, combining the preceding identities with \eqref{eq:approx-delta-f} and \eqref{eq:partition-delta-f-metric}, we obtain
    \begin{equation}\label{eq:diffuse-metric-control}
    \sup_{t\in[0,T]}
    \left\|
    \frac1\rho
    \int_{E_\rho\cap[0,t]}\Delta f(s)\,ds
    -
    \int_0^t\Delta f(s)\,ds
    \right\|_{L^2_{\mu_0}}
    \longrightarrow0.
    \end{equation}
    We now turn to the linearization of the flow. Using the flow representantions 
    \(
    \mu_t^\rho=(\Phi^\rho_{(0,t)})_\#\mu_0\) and 
    \(
    \bar\mu_t=(\Phi^{\bar u}_{(0,t)})_\#\mu_0,
    \)
    we obtain that
    \[
    W_2(\mu_t^\rho,\bar\mu_t)
    \leq
    \left\|
    \Phi^\rho_{(0,t)}-\Phi^{\bar u}_{(0,t)}
    \right\|_{L^2_{\mu_0}}.
    \]
    Then, by the regularity assumptions on \(f\), together with
    \eqref{eq:diffuse-metric-control}, we may subtract the equations for
    \(\Phi^\rho\) and \(\Phi^{\bar u}\), divide by \(\rho\), and apply Grönwall's
    lemma to obtain
    \[
    \sup_{t\in[0,T]}
    \left\|
    \frac{\Phi^\rho_{(0,t)}-\Phi^{\bar u}_{(0,t)}}{\rho}
    -\Psi(t)
    \right\|_{L^2_{\mu_0}}
    \longrightarrow0.
    \]
    Hence,
    \[
    \Phi^\rho_{(0,\cdot)}
    =
    \Phi^{\bar u}_{(0,\cdot)}
    +\rho\Psi+o(\rho)
    \qquad\text{in }C([0,T]; L^2_{\mu_0}).
    \]
\end{proof}

\begin{proof}[Proof of Proposition \ref{prop: taylor expansion costs}]
    The differentiability of \(\mcJ\) yields
    \[
    \mcJ(\mu_T^\rho)-\mcJ(\bar\mu_T)
    =
    \rho
    \int_{\mathbb R^d}
    \left\langle
    \nabla_\mu\mcJ(\bar\mu_T)
    \bigl(\Phi^{\bar u}_{(0,T)}(x)\bigr),
    \Psi(T,x)
    \right\rangle
    \,d\mu_0(x)
    +o(\rho).
    \]
    Similarly, using~\eqref{eq:Lyap-L} and differentiability of \(L^\epsilon\) in the
    measure variable,
    \[
    \begin{aligned}
    &\int_0^T L^\epsilon(t,u^\rho_t,\mu_t^\rho)\,dt
    -
    \int_0^T L^\epsilon(t,\bar u_t,\bar\mu_t)\,dt
    \\
    &\qquad=
    \rho\int_0^T\Delta L^\epsilon(t)\,dt
    \\
    &\qquad\quad+
    \rho\int_0^T\int_{\mathbb R^d}
    \left\langle
    \nabla_\mu L^\epsilon(t,\bar u(t),\bar\mu_t)
    \bigl(\Phi^{\bar u}_{(0,t)}(x)\bigr),
    \Psi(t,x)
    \right\rangle
    \,d\mu_0(x)\,dt
    +o(\rho).
    \end{aligned}
    \]
    The result holds after combining the two expansions.
\end{proof}

\section{Well-posedness of the adjoint equation}\label{sec: well-posedness adjoint}

\begin{proposition}[Well-posedness of the adjoint equation]\label{prop:well-posedness-adjoint}
    Assume [A\ref{hyp:dynamic-non-smooth}]-[A\ref{hyp: assumptions lagrangian}]. Let
    \(\alpha^*\in[0,1]\), \(\lambda^*\in\mathcal M_+([0,T])\), and let
    \[
    \zeta_\mcJ^*\in\partial_C\mcJ(\mu_T^*),
    \qquad
    \zeta_L^*(t)\in
    \partial_C L(t,u_t^*,\cdot)(\mu_t^*)
    \quad\text{for a.e. }t\in[0,T],
    \]
    and
    \[
    \zeta_g^*(t)\in\partial_C g(\mu_t^*)
    \quad\text{for }\lambda^*\text{-a.e. }t\in[0,T]
    \]
    be measurable selections. Define
    \[
    m^*(t)
    :=
    \int_{[0,t]}
    \zeta_g^*(s)\circ\Phi^*_{(0,s)}\,\lambda^*(ds).
    \]
Then the backward Stieltjes equation
\[
\left\{
\begin{aligned}
-dP^*(t,x)
&=
\Bigg[
D_x f[\mu_t^*]
\bigl(t,u_t^*,\Phi^*_{(0,t)}(x)\bigr)^\top
P^*(t,x)
\\
&\quad+
\int_{\R^d}
\nabla_\mu f[\mu_t^*]
\bigl(t,u_t^*,\Phi^*_{(0,t)}(y)\bigr)
\bigl(\Phi^*_{(0,t)}(x)\bigr)^\top
P^*(t,y)\,d\mu_0(y)
\\
&\quad-
\alpha^*
\zeta_L^*(t)
\bigl(\Phi^*_{(0,t)}(x)\bigr)
\Bigg]\,dt
-
dm^*(t,x),
\\
P^*(T,x)
&=
-\alpha^*
\zeta_\mcJ^*
\bigl(\Phi^*_{(0,T)}(x)\bigr)
\end{aligned}
\right.
\]
admits a unique solution
\[
P^*\in BV\bigl([0,T];L^2_{\mu_0}\bigr).
\]
Moreover,
\[
P^*-m^*
\in
\AC\bigl([0,T];L^2_{\mu_0}\bigr).
\]
\end{proposition}

\begin{proof}
    Since \(t\mapsto\mu_t^*\) is continuous, the set
    \[
    \{\mu_t^*:t\in[0,T]\}
    \]
    is compact in \(\mathscr P_2(\R^d)\). Since \(\mcJ\) and \(g\) are
    locally Lipschitz and \(L(t,u,\cdot)\) is locally Lipschitz uniformly
    with respect to \((t,u)\), a finite covering argument and
    Proposition~\ref{prop:basic-properties-clarke-subdiff} yield constants
    \(C_\mcJ,C_L,C_g>0\) such that
    \[
    \|\zeta_\mcJ^*\|_{L^2_{\mu_T^*}}\leq C_\mcJ,
    \qquad
    \|\zeta_L^*(t)\|_{L^2_{\mu_t^*}}\leq C_L
    \quad\text{for a.e. }t\in[0,T],
    \]
    and
    \[
    \|\zeta_g^*(t)\|_{L^2_{\mu_t^*}}\leq C_g
    \quad\text{for }\lambda^*\text{-a.e. }t\in[0,T].
    \]
    Using
    \((\Phi^*_{(0,t)})_\#\mu_0=\mu_t^*\), these estimates give
    \[
    \zeta_L^*(\cdot)\circ\Phi^*_{(0,\cdot)}
    \in L^\infty\bigl(0,T;L^2_{\mu_0}\bigr)
    \]
    and
    \[
    \operatorname{Var}_{L^2_{\mu_0}}(m^*;[0,T])
    \leq C_g\lambda^*([0,T]).
    \]
    Hence
    \(
    m^*\in BV\bigl([0,T];L^2_{\mu_0}\bigr).
    \)
    We first observe that [A\ref{hyp:dynamic-non-smooth}.ii] implies that
    \begin{equation}\label{eq:derivative-bounds-adjoint}
    \|D_xf[\mu](t,u,x)\|
    \leq \ell_f(t)
    \end{equation}
    and by
    Definition~\ref{def: differentiability on Wasserstein}, 
    \begin{equation}\label{eq:measure-derivative-bound-adjoint}
        \left\|
    \nabla_\mu f[\mu](t,u,x)(\cdot)^\top p
    \right\|_{L^2_{\mu}}
    \leq
    \ell_f(t)|p| \quad \forall p \in \R^d.
    \end{equation}
    For a.e. \(t\in[0,T]\), define
    \(\mathcal A_t:L^2_{\mu_0}\to L^2_{\mu_0}\) by
    \[
    \begin{aligned}
    (\mathcal A_tQ)(x)
    &:=
    D_x f[\mu_t^*]
    \bigl(t,u_t^*,\Phi^*_{(0,t)}(x)\bigr)^\top Q(x)
    \\
    &\quad+
    \int_{\R^d}
    \nabla_\mu f[\mu_t^*]
    \bigl(t,u_t^*,\Phi^*_{(0,t)}(y)\bigr)
    \bigl(\Phi^*_{(0,t)}(x)\bigr)^\top
    Q(y)\,d\mu_0(y).
    \end{aligned}
    \]
    By \eqref{eq:derivative-bounds-adjoint},
    \eqref{eq:measure-derivative-bound-adjoint}, Minkowski's inequality,
    and
    \((\Phi^*_{(0,t)})_\#\mu_0=\mu_t^*\), we obtain
    \[
    \begin{aligned}
    \|\mathcal A_tQ\|_{L^2_{\mu_0}}
    &\leq
    \ell_f(t)\|Q\|_{L^2_{\mu_0}}
    \\
    &\quad+
    \int_{\R^d}
    \left\|
    \nabla_\mu f[\mu_t^*]
    \bigl(t,u_t^*,\Phi^*_{(0,t)}(y)\bigr)
    \bigl(\Phi^*_{(0,t)}(\cdot)\bigr)^\top
    Q(y)
    \right\|_{L^2_{\mu_0}}
    \,d\mu_0(y)
    \\
    &\leq
    \ell_f(t)\|Q\|_{\mu_0}
    +
    \ell_f(t)
    \int_{\R^d}|Q(y)|\,d\mu_0(y)
    \\
    &\leq
    2\ell_f(t)\|Q\|_{L^2_{\mu_0}}.
    \end{aligned}
    \]
    
Set
\(
Q:=P^*-m^*
\).
The Stieltjes equation is equivalent to
\[
\begin{aligned}
Q(t)
&=
-\alpha^*
\zeta_\mcJ^*\circ\Phi^*_{(0,T)}
-
m^*(T)
\\
&\quad+
\int_t^T
\mathcal A_s\bigl(Q(s)+m^*(s)\bigr)\,ds
-
\alpha^*
\int_t^T
\zeta_L^*(s)\circ\Phi^*_{(0,s)}\,ds.
\end{aligned}
\]
    Let \(\mathcal T\) denote the right-hand side. Since
    \(m^*\) is bounded and \(\ell_f\in L^1(0,T)\), the mapping
    \(\mathcal T\) sends
    \(C([0,T];L^2_{\mu_0})\) into itself.
    For \(\lambda>0\), equip \(C([0,T];L^2_{\mu_0})\) with the equivalent
    norm
    \[
    \|Q\|_\lambda
    :=
    \sup_{t\in[0,T]}
    \exp\left(
    -2\lambda\int_t^T\ell_f(s)\,ds
    \right)
    \|Q(t)\|_{\mu_0}.
    \]
    For \(Q_1,Q_2\in C([0,T];L^2_{\mu_0})\), the estimate on
    \(\mathcal A_t\) gives
    \[
    \|\mathcal T(Q_1)(t)-\mathcal T(Q_2)(t)\|_{\mu_0}
    \leq
    2\int_t^T
    \ell_f(s)\|Q_1(s)-Q_2(s)\|_{\mu_0}\,ds.
    \]
    Multiplying both sides by
    \(
    \exp\left(-2\lambda\int_t^T\ell_f(r)\,dr\right)
    \)
    and using the definition of \(\|\cdot\|_\lambda\), we obtain
    \[
    \begin{aligned}
    &e^{-2\lambda\int_t^T\ell_f(r)\,dr}
    \|\mathcal T(Q_1)(t)-\mathcal T(Q_2)(t)\|_{\mu_0}
    \\
    &\qquad\leq
    2\|Q_1-Q_2\|_\lambda
    \int_t^T
    \ell_f(s)
    e^{-2\lambda\int_t^s\ell_f(r)\,dr}\,ds
    \\
    &\qquad\leq
    \frac{1}{\lambda}\|Q_1-Q_2\|_\lambda.
    \end{aligned}
    \]
    Taking the supremum over \(t\in[0,T]\), we obtain
    \[
    \|\mathcal T(Q_1)-\mathcal T(Q_2)\|_\lambda
    \leq
    \frac{1}{\lambda}
    \|Q_1-Q_2\|_\lambda.
    \]
    Thus, for \(\lambda>1\), \(\mathcal T\) is a strict contraction and
    admits a unique fixed point
    \[
    Q\in C\bigl([0,T];L^2_{\mu_0}\bigr).
    \]
    
    Since \(T<\infty\) and \(\ell_f\in L^2(0,T)\), we have
    \(\ell_f\in L^1(0,T)\). The integral equation then yields
\[
\partial_tQ(t)
=
-\mathcal A_t\bigl(Q(t)+m^*(t)\bigr)
+
\alpha^*
\zeta_L^*(t)\circ\Phi^*_{(0,t)}
\]
with
\[
\partial_tQ
\in
L^1\bigl(0,T;L^2_{\mu_0}\bigr).
\]
Consequently,
\[
Q=P^*-m^*
\in
\AC\bigl([0,T];L^2_{\mu_0}\bigr).
\]
Since
\(
P^*=Q+m^*
\)
and
\(
m^*\in BV\bigl([0,T];L^2_{\mu_0}\bigr)
\),
it follows that
\[
P^*\in BV\bigl([0,T];L^2_{\mu_0}\bigr).
\]
Uniqueness of \(P^*\) follows from the uniqueness of \(Q\).
\end{proof}

\section*{Statements \& Declarations}
\paragraph{Conflict of interest} The authors have no conflict of interest.

\bibliographystyle{plainurl}
\bibliography{biblio}

\end{document}